\documentclass[a4paper,12pt,dvipdfmx]{amsart}
\usepackage{amssymb, amsmath, amsthm, mathrsfs, bm}
\usepackage{enumitem}
\usepackage{url}

\usepackage[all, warning]{onlyamsmath}
\usepackage{graphicx}
\usepackage[ruled, lined]{algorithm2e}
\newtheorem{theorem}{Theorem}[section]

\newtheorem{lemma}{Lemma}[section]
\newtheorem{proposition}{Proposition}[section]

\newtheorem{assumption}{Assumption}[section]
\newtheorem{remark}{Remark}[section]

\DeclareMathOperator{\bd}{bd}
\DeclareMathOperator{\inte}{int}
\usepackage{tikz}
\usepackage{pgfplots}
\pgfplotsset{compat=newest}
\makeatletter
\global\let\tikz@ensure@dollar@catcode=\relax
\makeatother

\title[Convergence rate of APM]{Convergence rate of alternating projection method for the intersection of an affine subspace and the second-order cone}

\author[H. Ochiai]{Hiroyuki Ochiai}
\address[H. Ochiai]{Institute of Mathematics for Industry, Kyushu University, 744 Motooka, Nishi-ku, Fukuoka 819-0395, Japan}
\email{\ttfamily{ochiai@imi.kyushu-u.ac.jp}}
\author[Y. Sekiguchi]{Yoshiyuki Sekiguchi}
\address[Y. Sekiguchi]{Graduate School of Marine Science and Technology, Etchujima 2-1-8, Koto-ku, Tokyo 135-8533, Japan}
\email{\ttfamily{yoshi-s@kaiyodai.ac.jp}}
\author[H. Waki]{Hayato Waki}
\address[H. Waki]{Institute of Mathematics for Industry, Kyushu University, 744 Motooka, Nishi-ku, Fukuoka 819-0395, JAPAN}
\email{\ttfamily{waki@imi.kyushu-u.ac.jp}}

\begin{document}

\maketitle
\begin{abstract}
We study the convergence rate of the alternating projection method (APM) applied to the intersection of an affine subspace and the second-order cone. We show that when they intersect non-transversally, the convergence rate is $O(k^{-1/2})$, where $k$ is the number of iterations of the APM. In particular, when the intersection is not at the origin or forms a half-line with the origin as the endpoint, the obtained convergence rate can be exact because a lower bound of the convergence rate is evaluated. These results coincide with the worst-case convergence rate obtained from the error bound discussed in \cite{Borwein2014} and \cite{Drusvyatskiy2017}. Moreover, we consider the convergence rate of the APM for the intersection of an affine subspace and the product of two second-order cones. We provide an example that the worst-case convergence rate of the APM is better than the rate expected from the error bound for the example. 
\end{abstract}
\section{Introduction and summary}

The alternating projection method (APM) is a well-known algorithm for finding a point in the intersection of multiple sets. There is a vast amount of literature on the theory and applications of the APM, see, e.g., \cite{Bauschke1996} and \cite{Noll2016} and the references therein. For example, it is shown in \cite{Gubin1966} that the sequences generated by the APM for the intersection of multiple convex sets converge to a point of the intersection. It is also known  in \cite[Theorem 2]{Gubin1966} that the convergence rate is linear when multiple convex sets intersect transversally. It is proved in \cite{Lewis2008} that the APM for the intersection of non-convex sets converges with a linear rate. 

When multiple convex sets intersect non-transversally, the convergence rate of the sequence generated by the APM is expected to be sublinear. The convergence rate of the APM has been analyzed in \cite{Borwein2014}. In particular, the authors proved that H{\"o}lder regularity holds for convex sets composed of convex polynomials and provided the convergence rate by using the regularity. In \cite{Luke2018} and \cite{Noll2021}, the convergence rate of the APM for non-convex sets is given by extending the regularity theory. 

In these studies, via the regularity theory, inequalities called error bounds play an essential role in analyzing the worst-case convergence rate.  However, the error bounds cannot be used to evaluate the convergence rate from below. This is the reason why the evaluation of {\itshape the exact convergence rate} is considered to be complicated even under the assumption of convexity. Indeed, there are a few results on the exact convergence rate of the APM. For instance, in \cite{Ochiai2021}, the exact convergence rate of the APM is obtained for particular sets. In more detail, the authors analyze the APM for the intersection of  a line or linear subspace and a convex set determined by a convex polynomial and  derive the exact convergence rate from the algebraic recursion formula that determines the sequence generated by the APM. It is also introduced that the convergence rate of the APM for the intersection of a line and  a convex set defined by two convex polynomials may vary depending on the initial points. 

\subsection*{Purpose and contributions}
We discuss the convergence rate of the APM for the intersection of an affine subspace and the second-order cone and find the exact convergence rate. As mentioned, it is known  that the APM converges at a linear rate when the affine subspace and the second-order cone intersect transversally. Therefore, we are interested in the convergence rate of the APM when they intersect non-transversally. 

A study related to this problem setting is known in \cite{Drusvyatskiy2017, Liu2022}. It is shown in \cite[Theorems 2.2 and 2.4]{Drusvyatskiy2017} that for the intersection of an affine subspace $H$ and the positive semidefinite cone $\mathbb{S}^n_+$, if the pair $(H, \mathbb{S}^n_+)$ satisfies the $\gamma$-H{\"o}lder regularity holds for some $\gamma\in (0, 1)$, then the sequence of the APM for the intersection converges with the sublinear rate $O(k^{-1/(2\gamma^{-1}-2)})$, where $k$ is the number of iterations of the APM, and that the convergence rate is linear when $\gamma=1$. In particular, the authors mention that  $\gamma = 1/2^d$ holds, where $d$ is the minimal iteration number of facial reduction \cite{Borwein1980, Borwein1981} and is called {\itshape singularity degree}. This equality is obtained from the H{\"o}lder regularity in \cite{Sturm2000}. Since the second-order cone can be described using positive semidefinite matrices (see, e.g., \cite[page 7]{Alizadeh2003}),  the result in \cite[Theorem 2.4]{Drusvyatskiy2017} is directly applicable to this case. Hence we can see that the APM converges at the worst-case rate $O(k^{-1/(2^{d+1}-2)})$, where $d$ is the singularity degree determined by this intersection when an affine subspace and the second-order cone intersect non-transversally. Furthermore, a slightly improved convergence rate is provided in \cite[Theorem 5.5]{Liu2022} by introducing the notion of consistent error bound functions.

The contribution of this paper is to provide a more precise analysis of the convergence rate of the APM for this problem setting. There are three possible cases where an affine subspace and the second-order cone intersect non-transversally: (i) the intersection is at the origin only, (ii) the intersection is at a point other than the origin, and (iii) the intersection of an affine subspace and the second-order cone is a half-line  with the origin as the endpoint. 

In case (i), the sequence generated by the APM converges linearly at the origin. Indeed, the linear regularity holds in this case. (It is also mentioned in \cite[Section 2]{Sturm2000} that linear regularity holds even when the intersection of the semidefinite cone and an affine subspace is at the origin only.) Combining this fact and \cite[Proposition 4.2]{Borwein2014}, we can also show that APM converges linearly. This fact is also obtained from \cite[Theorem 2.2]{Drusvyatskiy2017}.

We evaluate more precise convergence rates in cases (ii) and (iii). For this, we derive algebraic recursions on the sequence generated by the APM. Our result is summarized in Theorem ~\ref{thm:convrate}, which is described in more detail as follows: 
\begin{enumerate}[label=(\roman*)]
\setcounter{enumi}{1}
    


    \item We prove that the sequence generated by the APM converges to a point with the exact rate $\Theta(k^{-1/2})$. More precisely, the sequence $\{\bm{u}_k\}$ generated by the APM satisfies the following inequalities: there exists $d_1, d_2>0$ and $\tilde{k}\in\mathbb{N}$ such that for all $k\ge \tilde{k}$, 
    \[
    d_1k^{-1/2}\le \|\bm{u}_k -\bm{u}_*\| \le d_2 k^{-1/2},  
    \]
    where $\bm{u}_*$ is a convergent point of the sequence $\{\bm{u}_k\}$. Note that we can also derive a lower bound for the convergence rate. We call such a convergence rate {\itshape the exact convergence rate}. Since the singularity degree is 1, the result coincides with the worst-case convergence rate obtained in \cite[Theorem 2.2]{Drusvyatskiy2017}.
    
    \item We prove that for certain initial points, the sequence reaches the origin in one iteration, and for other initial points, the sequence converges to a point on the half-line with the exact rate $\Theta(k^{-1/2})$. 
    Our exact convergence rate is equal to the worst-case convergence rate obtained in \cite[Proposition 4.2]{Borwein2014} and \cite[Theorem 2.2]{Drusvyatskiy2017}. 

    
\end{enumerate}

 Another contribution is to consider the convergence rate of the APM for the intersection of an affine subspace and the  product set of two second-order cones. In particular, we provide two examples. The singularity degrees of those examples are two, and the worst-case convergence rate $O(k^{-1/6})$ is tight in the first example, which coincides with the worst-case convergence rate induced from the error bound with the singularity degree, while it is not tight in the second example. Indeed, the convergence rate in the second example is $O(k^{-1/2})$. The convergence rate of the second example is better than the worst-case convergence rate expected from the error bound with the singularity degree. 

 \subsection*{Future work}
 As we will see in the examples of the convergence rate of the APM for the intersection of an affine subspace and the product set of two second-order cones, a sequence of the APM can be generated by more than one recurrent formula. Unlike the case of the intersection of an affine subspace with a single second-order cone, this makes the analysis more difficult. Characterization of the affine subspace $H$ in such a way that the estimation of the convergence rate by the error bound determined by the singularity degree is tight remains as future work. Also, in the second example, a more precise analysis is used to obtain the exact convergence rate of the APM. The generalization of the analytical methods used for the examples is one of our future tasks. 

\subsection*{Organization}
We introduce the basic notation on the second-order cone, a brief introduction of singularity degree, and inequalities on the sequences for the convergence rate in Section~\ref{sec:preliminary}. In Section~\ref{sec:convergence}, we provide a proof of Theorem~\ref{thm:convrate}, summarizing our main result. For this, we consider three cases: (i) an affine subspace and the second-order cone intersect at the origin only, (ii) an affine subspace and the second-order cone intersect at a point other than the origin non-transversally, and (iii)  an affine subspace and the second-order cone intersect a half-line with the origin as the endpoint. We provide two examples where the convergence rate for the intersection of an affine subspace and the  product of two second-order cones is discussed in Section~\ref{sec:product}. Since the analysis is complicated, the details and proofs of the convergence rate of the APM are provided in Appendices~\ref{app:convrateEx1} and \ref{app:property}.


\section{Preliminary}\label{sec:preliminary}
\subsection{Notation and symbols}
For $a\in\mathbb{R}$ and $\bm{b}\in\mathbb{R}^n$, we abbreviate $\begin{bmatrix}a\\ \bm{b}\end{bmatrix}$ as $(a; \bm{b})$. Similarly, for $\bm{x}\in\mathbb{R}^n$ and $\bm{y}\in\mathbb{R}^m$, we write $\begin{bmatrix}\bm{x}\\ \bm{y}\end{bmatrix}$ as $(\bm{x}; \bm{y})$. We also define $\hat{\bm{x}}=(x_0; -\bm{x}_1)$ for $\bm{x} = (x_0; \bm{x}_1)$. 

We define the set $\mathcal{K}_{n+1}$ as 
\[
\mathcal {K}_{n+1} = \left\{
(x_0; \bm{x}_1)\in\mathbb{R}^{n+1} : \|\bm{x}_1\|_2\le x_0
\right\}
\]
We call it the second-order cone. We also write $\mathcal{K}$ if the dimension can be determined from the context. We write $\inte\mathcal{K}$ for the interior of the second-order cone $\mathcal{K}$ and $\bd\mathcal{K}$ for its boundary. In other words, 
\begin{align*}
    \inte\mathcal{K}&=\left\{(x_0; \bm{x}_1)\in\mathbb{R}^{n+1} : \|\bm{x}_1\|_2 < x_0\right\}, 
    \bd\mathcal{K}&=\left\{(x_0; \bm{x}_1)\in\mathbb{R}^{n+1} : \|\bm{x}_1\|_2= x_0\right\}. 
\end{align*}

The following lemma shows a property of points on the boundary of the second-order cone. The proof is given in Appendix~\ref{sec:proofoflemma}. 

\begin{lemma}\label{lemma:socp0}
Assume that $\bm{x}\in\bd\mathcal{K}\setminus\{\bm{0}\}$. If $\bm{y}\in\mathcal{K}$ is orthogonal to $\bm{x}$, then $\bm{y}$ is formed as $\bm{y} = \alpha \hat{\bm{x}}$ for some $\alpha \ge 0$. 
\end{lemma}

For any $(x_0; \bm{x}_1)\in\mathbb{R}^{n+1}$, we can write $\bm{x}$ as follows:
\begin{align}
\label{spectraldec}
(x_0; \bm{x}_1) &= \lambda_1 \bm{e}_1 + \lambda_2 \bm{e}_2, \\
\nonumber \lambda_1&= \frac{x_0 +\|\bm{x}_1\|_2}{\sqrt{2}}, \lambda_2 = \frac{x_0-\|\bm{x}_1\|_2}{\sqrt{2}}, \\
\nonumber \bm{e}_1 &= \frac{1}{\sqrt{2}}\left(
1; \frac{\bm{x}_1}{\|\bm{x}_1\|_2}\right), \bm{e}_2 =\frac{1}{\sqrt{2}}\left(
1; -\frac{\bm{x}_1}{\|\bm{x}_1\|_2}\right). 
\end{align}
This is called the spectral decomposition of $\bm{x}$ (with respect to $\mathcal{K}$). It is known that the projection $P_{\mathcal{K}}(\bm{x})$ on $\mathcal{K}$ can be written as follows : 

\begin{theorem}(\cite[Theorem 3.3.6]{Bauschke1990})\label{thm:pprojsocp}
For any $\bm{x}=(x_0; \bm{x}_1)\in\mathbb{R}^{n+1}$, the following holds: 
\[
P_{\mathcal{K}}(\bm{x}) = \left\{
\begin{array}{cl}
(x_0; \bm{x}_1) & \mbox{ if } x_0\ge \|\bm{x}_1\|_2 \ (i.e., \bm{x}\in\mathcal{K}), \\
(0; \bm{0}) & \mbox{ if } -x_0\ge \|\bm{x}_1\|_2 \ (i.e., -\bm{x}\in\mathcal{K}), \\
\lambda_1 \bm{e}_1 & \mbox{otherwise} 
\end{array}
\right.
\]
where $\lambda_1$ and $\bm{e}_1$ correspond to the first component of the spectral decomposition \eqref{spectraldec} of $(x_0; \bm{x}_1)$.
\end{theorem}

\subsection{Intersection of an affine subspace and the second-order cone}

In this paper, except for Section ~\ref{sec:product}, we assume the following. 
\begin{assumption}\label{assumption}
\leavevmode
\makeatletter
\@nobreaktrue
\makeatother
\begin{itemize}
\item The dimension $n+1$ of $\mathcal{K}$ is more than 2.  
\item The intersection of  an affine subspace $H$  and the second-order cone $\mathcal{K}$ is nonempty, but $H\cap \inte\mathcal{K}$ is empty. 
\end{itemize}
\end{assumption}

If $n+1\le 2$, then $\mathcal{K}$ is a polyhedral cone. Thus,  by combining \cite[Proposition 4.2]{Borwein2014} with Hoffman's error bound, see, e.g., \cite[Proposition 5.25]{Bauschke1996}, we can prove that the convergence rate of the APM is linear. 

When the intersection $H\cap \mathcal{K}$ consists of a single point $\bm{u}_*$,  we can write the affine subspace $H$ as follows: 
\[
H = \left\{
\bm{u}_* + \bm{B}\bm{t} : \bm{t}\in\mathbb{R}^p
\right\}, 
\]
where, $\bm{B}\in\mathbb{R}^{(n+1)\times p}$ satisfies $\bm{B}^T\bm{B} = \bm{I}_p$. Moreover, we write $\bm{b}^T\in\mathbb{R}^{1\times p}$ for the vector consisting of the first row of the matrix $\bm{B}$ and $\bm{B}_1\in\mathbb{R}^{n\times p}$ for the matrix formed by excluding the first row $\bm{b}^T$ from $\bm{B}\in\mathbb{R}^{(n+1)\times p}$. Thus, we can write,  
\[
\bm{B} = \begin{bmatrix}
\bm{b}^T\\
\bm{B}_1
\end{bmatrix}
\]
Note that for $p=1$, we do not use the bold font for the first row of $\bm{B}$, i.e., $\bm{B} = \begin{bmatrix}
b\\
\bm{B}_1
\end{bmatrix}$. 

When the intersection $H\cap\mathcal{K}$ consists of multiple points, i.e., $H\cap \mathcal{K} =\{\alpha \bm{d} : \alpha \ge 0\}$, where $\bm{d} \in\mathcal{K}\setminus\{\bm{0}\}$, then we write the affine subspace $H$ as $H=\{\bm{B}\bm{t} : \bm{t}\in\mathbb{R}^p\}$. The assumptions and notations about $\bm{B}$ are the same as above. 

We use the following facts to analyze the convergence rate of the APM. Their proofs are given in Appendix \ref{sec:proofoflemma}. 
\begin{lemma}\label{lemma:matrix}
The eigenvalues of the matrix $\bm{I}_p-2\bm{b}\bm{b}^T$ are $1$ and $1-2\|\bm{b}\|^2$ and their (algebraic) multiplicity are $(p-1)$ and $1$, respectively.
\end{lemma}
\begin{lemma}\label{lemma:socp}
Under Assumption~\ref{assumption}, the following hold. 
\begin{enumerate}[label=(P\arabic*)]

\item \label{lemma:P2}  $H\cap\bd\mathcal{K}=\{\bm{u}_*\}$ if and only if  $\bm{B}^T\hat{\bm{u}}_*=\bm{0}$ and $2\|\bm{b}\|^2 < 1$
\item \label{lemma:P4} $H\cap\bd\mathcal{K}=\{\alpha\bm{d} : \alpha \ge 0\}$ for some $\bm{d}\in\bd\mathcal{K}\setminus\{\bm{0}\}$ if and only if $\|\bm{b}\|^2=1/2$. 
\item \label{lemma:P3} If $H\cap\bd\mathcal{K}=\{\bm{u}_*\}$, then for any $\bm{u}\in H\setminus\{\bm{u}_*\}$, $-\hat{\bm{u}} \not\in\mathcal{K}$. 
\item \label{lemma:P6} If $H\cap\bd\mathcal{K}=\{\bm{u}_*\}$, then for any $\bm{u}\in H\setminus\{\bm{u}_*\}$, $P_{\mathcal{K}}(\bm{u})\neq \bm{u}_*$. 
\item \label{lemma:P5} If $H\cap\bd\mathcal{K}=\{\alpha\bm{d} : \alpha \ge 0\}$ for some $\bm{d}\in\bd\mathcal{K}\setminus\{\bm{0}\}$, then for any $\bm{u}\in H\setminus\{\alpha \bm{d} : \alpha\in\mathbb{R}\}$, $-\hat{\bm{u}} \not\in\mathcal{K}$. 
\end{enumerate}
\end{lemma}

\begin{remark}
\ref{lemma:P3} asserts that $\bm{u}\in H\setminus\{\bm{u}_*\}$ satisfies $-u_0 < \|\bm{u}_1\|$ if the assumption of Lemma \ref{lemma:socp} is satisfied. Thus, it follows from Theorem \ref{thm:pprojsocp} that the projection to $\mathcal{K}$ of $\bm{u}\in H\setminus\{\bm{u}_*\}$ will be onto the first component $\lambda_1\bm{e}_1$ of the spectral decomposition of $\bm{u}$. Similarly, it follows from \ref{lemma:P5} that the projection to $\mathcal{K}$ of $\bm{u}\in H\setminus\{\alpha \bm{d} : \alpha\in\mathbb{R}\}$ will also be on to  first component $\lambda_1\bm{e}_1$ of the spectral decomposition of $\bm{u}$. Therefore, if the APM does not terminate in finitely many iterations, the projection of $\bm{u}$ onto $\mathcal{K}$ is the first component of its spectral decomposition. 

In addition, \ref{lemma:P6} implies that the APM does not terminate in finitely many iterations when the intersection of an affine subspace and the second-order cone is a singleton. Indeed, if the APM terminates finitely many iterations, then the projection of a $\bm{u}\in H$ to $\mathcal{K}$ must be $\bm{u}_*$. However, \ref{lemma:P6} ensures that  such $\bm{u}\in H\setminus\{\bm{u}_*\}$ does not exist. On the other hand, when the intersection is a half-line, the APM can terminate in finitely many iterations by choosing a specific initial point. See the subsection~\ref{sec:Case3} for the detail.
\end{remark}

\subsection{Facial structure of the second-order cone and the singularity degree}
We introduce the facial structure of the second-order cone $\mathcal{K}$ and the singularity degree. For this, we also explain facial reduction for the intersection of an affine subspace $H$ and $\mathcal{K}$ (or the product of second-order cones) because its singularity degree is defined as the minimum number of iterations of facial reduction. 

The concepts and symbols of convex analysis, which are often used to explain facial reduction, are introduced. For a nonempty set $S\subset\mathbb{R}^n$, $\bm{x}\in S$ is a relative interior of $S$ if the intersection of a neighborhood centered at $x$ and the affine hull of $S$ is contained in $S$. The set of all relative interiors of $S$ is called the relative interior of $S$, and is denoted by $\mathrm{ri}(S)$. For a convex set $C\subseteq\mathbb{R}^n$, a convex subset $F$ of $C$, we say that $F$ is a {\itshape face} of $C$ if $\bm{x}_1, \bm{x}_2\in C$ and the intersection of the open line segment $(\bm{x}_1, \bm{x}_2)$ and $F$ is nonempty implies that $\bm{x}_1, \bm{x}_2$ are both in $F$. For a face $F$ of a closed convex cone, $F^*$ and $F^\perp$ are defined as 
\begin{align*}
F^*&=\{\bm{y} : \bm{x}^T\bm{y} \ge 0 \ (\forall \bm{x}\in F)\}, F^\perp=\{\bm{y} : \bm{x}^T\bm{y} = 0 \ (\forall \bm{x}\in F)\}.
\end{align*}
For a nonempty convex subset $S$ of a convex cone $K\subset\mathbb{R}^n$, the {\itshape minimal face} of $K$ containing $S$ is defined as the intersection of all faces of $K$ that contain $S$, and is denoted by $\mathrm{face}(S, K)$. Note that the minimal face of is the smallest face of $K$ intersecting the relative interior of $S$, and $\mathrm{face}(S, K) = \mathrm{face}(\{\bm{x}\}, K)$ for any $\bm{x}\in\mathrm{ri}(S)$. 

In general, any nonempty face of the second-order cone $\mathcal{K}$ is either $\{\bm{0}\}$, itself or 
$C(\bm{x}) :=\{\alpha \bm{x} : \alpha \ge 0\}$, where $\bm{x}=(x_0; \bm{x}_1)\in\bd\mathcal{K}\setminus\{\bm{0}\}$. Also, when the set $K$ is the product of second-order cones, its face is the direct product of faces of the second-order cones. 

Let $K$ be the direct product of $r$ second-order cones with each length $n_j$ \ ($j=1, \ldots, r$). Recall that an affine subspace $H$ is formulated as $H=\{\bm{u}_* + \bm{B}\bm{t} : \bm{t}\in\mathbb{R}^p\}$. We assume that $H\cap K$ is nonempty. 
Facial reduction finds a minimal face $F_* :=\mathrm{face}(H\cap K, K)$, 
which satisfies $H\cap K = H \cap F_*$ and $H\cap \mathrm{ri}(F_*) \neq \emptyset$. The following theorem plays an essential role in facial reduction. We call $\bm{y}$ in the following theorem {\itshape reducing certificate}. 

\begin{theorem}\label{thm:fr}(\cite[Lemma 28.4]{Pataki2013})
Let $F$ be a closed convex cone. $H\cap \mathrm{ri}(K)$ is empty if and only if there exists $\bm{y}\in F^*\setminus F^\perp$ such that $\bm{u}_*^T\bm{y}=0$ and $\bm{B}^T\bm{y}=\bm{0}$. Moreover, if such $\bm{y}$ exists, then the set $G:=F\cap \{\bm{y}\}^\perp$ is strictly smaller face of $F$. 
\end{theorem}

Facial reduction starts with $K$ in Theorem \ref{thm:fr} as $F$. If no vectors $\bm{y}$ as in Theorem \ref{thm:fr}, then $K$ is the minimal face, and it  terminates. Otherwise, Theorem \ref{thm:fr} guarantees that $H\cap K$ lies in the smaller face $G=K\cap\{\bm{y}\}^\perp$. We replace $F$ in Theorem \ref{thm:fr} by $G$ and repeat this procedure. Since the facial reduction finds the minimal face or a face with a smaller dimension  (i.e., the dimension of the affine hull of the face) in each iteration, it terminates in finitely many iterations. The minimal number of iterations that are necessary for the facial reduction to terminate is called the singularity degree of the intersection of $H\cap K$. When $K$ is the direct product of $r$ second-order cones, the singularity degree is at most $2r$.  In \cite[Proposition 24]{Lourenco2018}, an example given of intersections of an affine subspace and the second-order cone such that the singularity degree is $r$. 

The following proposition provides the singularity degree of $H\cap K$ when $K$ is a single second-order cone, i.e., $K=\mathcal{K}$. We postpone the proof in Appendix~\ref{sec:proofOfsd}. 
\begin{proposition}\label{prop:fr}
Consider the intersection of an affine subspace $H\subset\mathbb{R}^{n+1}$ and a second-order cone $\mathcal{K}$. Assume that the intersection is nonempty and that $\bm{u}_*\in\bd\mathcal{K}$. The singularity degree of $H\cap \mathcal{K}$ is at most one. 
\end{proposition}

\subsection{Lemma for exact convergence rate}
We use the following lemma to find the convergence rate of this case. The proof is given in Appendix~\ref{sec:proofOfLemma31}. 

\begin{lemma}\label{lemma1}
\begin{enumerate}[label=(L\arabic*)]
\item \label{L2} Consider a sequence  $\{x_k\}\subset\mathbb{R}$ that satisfies  $x_k>0$ \ $(k=0, 1, \ldots)$, $x_k\to0 \ (k\to\infty)$. Assume that there exist $C>0$, $\tilde{k}, q\in\mathbb{N}$, and a convergent power series $h$ such that 
\begin{align}\label{ineqL2}
x_{k+1}&\le x_k(1-Cx_{k}^q + x_k^{q+1}h(x_k)) \ (k\ge \tilde{k}),  
\end{align}
Then, $\displaystyle\limsup_{k\to\infty} (qC)^{1/q}k^{1/q}x_k \le 1$. 
\item \label{L3} Consider a sequence $\{x_k\}\subset\mathbb{R}$ that satisfies $x_k>0$ \ $(k=0, 1, \ldots, )$, $x_k\to0 \ (k\to\infty)$.  Assume that there exist $C>0$, $\tilde{k}, q\in\mathbb{N}$ and a convergent power series $h$ such that 
\begin{align}\label{ineqL3}
x_{k+1}&\ge x_k(1-Cx_{k}^q + x_k^{q+1}h(x_k)) \ (k\ge \tilde{k}), 
\end{align}
Then, $\displaystyle\liminf_{k\to\infty} (qC)^{1/q}k^{1/q}x_k \ge 1$. 

\item\label{L4} Consider a sequence $\{x_k\}\subset\mathbb{R}$ that converges to $x_\infty$ and $\{w_k\}\subset\mathbb{R}$  defined as 
\[
w_{k+1} = w_k\left(1 - w_k^q f(x_k, w_k)\right) \ (k=0, 1, \ldots)
\]
with positive integer $q$ and a function $f : \mathbb{R}^2 \to\mathbb{R}$ which is continuous at $(x_\infty, 0)$. Assume that $w_k > 0$ for all $k$ and $w_k\to 0$ \ $(k\to\infty)$. Then, we have $\displaystyle\lim_{k\to\infty} 1/(k w_{k}^q) = qf(x_\infty, 0)$. In particular, if $f(x_\infty, 0)\neq 0$, then this implies $w_{k} = \Theta(k^{-1/q})$. 

\item\label{L5} Consider a sequence $\{a_k\}\subset\mathbb{R}$ that converges to $a_\infty$. In addition, we consider a sequence $\{x_k\}\subset \mathbb{R}$ that converges to $x_\infty$ and satisfies
\[
x_{k+1} - x_k = \frac{a_k}{k^2} \ (k=0, 1, \ldots).
\]
Then, we have $\displaystyle\lim_{k\to\infty} k(x_\infty - x_k) = a_\infty$. In particular, if $a_\infty \neq 0$, this implies $|x_\infty - x_k| = \Theta(k^{-1})$. 
\end{enumerate}
\end{lemma}

\section{Alternating projection method and the convergence rate}\label{sec:convergence}
Let $P_{\mathcal{K}}(\bm{u})$ be the projection of $\bm{u}$ onto $\mathcal{K}$ and $P_{H}(\bm{v})$ be the projection of $\bm{v}$ onto $H$. The APM generates a sequence $\{\bm{u}_k\}\subset\mathbb{R}^{n+1}$ for the initial point $\bm{u}_0\in H$ as follows: 
\[
\bm{u}_{k+1} = P_H(P_{\mathcal{K}}(\bm{u}_k)) \ (k=0, 1, \ldots, ). 
\]
The initial point $\bm{u}_0$ satisfies $\bm{u}_0\not\in\mathcal{K}$. Otherwise, we do not need to apply the APM to find a point in $H\cap\mathcal{K}$. 



Since the sequence $\{\bm{u}_k\}$ obtained by applying the APM satisfies $\bm{u}_k\in H$ \ $(k=0, 1, \ldots, )$, 
$\bm{u}_{k}=\bm{u}_* + \bm{B}\bm{t}_{k}$ \ $(k=0, 1, \ldots, )$ can be written. Therefore, $\|\bm{u}_k -\bm{u}_*\| = \|\bm{B}\bm{t}_{k}\| = \|\bm{t}_k\|$. This shows that the convergence rates of the sequences $\{\|\bm{u}_k -\bm{u}_*\|\}$ and $\{\|\bm{t}_k\|\}$ are the same. In the following, we discuss the convergence rate of $\{\|\bm{t}_k\|\}$. 
The following theorem summarizes the results obtained in this section: 

\begin{theorem}\label{thm:convrate}
Let $\{\bm{u}_k\}$ be a sequence generated by applying the alternating projection to the intersection of an affine subspace $H$ and the second-order cone $\mathcal{K}$.   
Under 
Assumption~\ref{assumption}, the following hold:
\begin{enumerate}[label=(\arabic*)]
\item \label{Conv1} If $H\cap\bd\mathcal{K}=\{\bm{0}\}$, it converges to $\bm{0}$ at a linear rate.
\item \label{Conv2} If $H\cap\bd\mathcal{K}=\{\bm{u}_*\}$ and $\bm{u}_{*}\neq \bm{0}$, it converges to $\bm{u}_*$ at $\Theta(k^{-1/2})$ rate.
\item\label{Conv3} If $H\cap\bd\mathcal{K}=C(\bm{d})$ for some $\bm{d}\in \bd\mathcal{K}\setminus\{\bm{0}\}$, it attains the origin in one iteration or it converges to a point in $C(\bm{d})$ other than the origin with $\Theta(k^{-1/2})$ rate.
\end{enumerate}
\end{theorem}
It follows from Lemma~\ref{lemma:socp} that we have $\bm{u}_k\not\in\mathcal{K}$ \ $(k=0, 1, \ldots)$ in the cases \ref{Conv1} and \ref{Conv2} of Theorem~\ref{thm:convrate},  Furthermore, from \ref{lemma:P3}  and \ref{lemma:P5} of Lemma~\ref{lemma:socp}, $P_{\mathcal{K}}(\bm{u}_k) = \lambda_1 \bm{e}_1$ at each iteration, where $\lambda_1\bm{e}_1$ is the first component of the spectral decomposition of $\bm{u}_k$. On the other hand, in \ref{Conv3} of Theorem~\ref{thm:convrate}, if we choose an initial point from the set $\{\alpha\bm{d}: \alpha\in\mathbb{R}\}$, the APM terminates in at most one iteration. 

\subsection{Case 1: $H\cap\bd\mathcal{K}=\{\bm{0}\}$}
When the affine subspace and the second-order cone intersect  non-transversally at the origin only, the linear regularity, which is stronger than the H{\"o}lder regularity with the exponent $1/2$, holds. This can be summarized as the following fact. Item \ref{Conv1} of Theorem~\ref{prop:linearreg} follows from this fact. 
\begin{proposition}\label{prop:linearreg}
Suppose that the intersection of the affine subspace $H$ and the second-order cone $\mathcal{K}$ is $\{\bm{0}\}$. Then, there exists $c>0$ such that for all $\bm{x}\in\mathbb{R}^{n+1}$, 
\begin{align}\label{eq:errorbound}
\mathrm{dist}(\bm{x}, H\cap \mathcal{K}) &\le c \left(
\mathrm{dist}(\bm{x}, H) + \mathrm{dist}(\bm{x},  \mathcal{K}) 
\right). 
\end{align}
Here $\mathrm{dist}(\bm{x}, A)=\inf\{\|\bm{x}-\bm{y}\| : \bm{y}\in A\}$ for any set $A\subset\mathbb{R}^{n+1}$, which means that the distance from the point $\bm{x}$ to the set $A$. 
\end{proposition}
When the set $A$ is a closed convex cone, 
we have $\mathrm{dist}(\alpha\bm{x}, A) =|\alpha|\mathrm{dist}(\bm{x}, A)$ for any $\alpha\in\mathbb{R}$. 

It should be noted that this is also mentioned in \cite[Section 2]{Sturm2000} for the intersection of the positive semidefinite cone and an affine subspace.
Hence, it follows from \cite[Proposition 4.2]{Borwein2014} or \cite[Theorem 2.2]{Drusvyatskiy2017} that the convergence rate of the APM is linear. 

\subsection{Case 2: $H\cap\bd\mathcal{K}=\{\bm{u}_*\}$ and $\bm{u}_{*}\neq \bm{0}$}


The following proposition analyzes one iteration of the APM. We postpone the proof in Appendix~\ref{sec:proofOfprop}. 
\begin{lemma}\label{prop:updaterule}
Let $\bm{u}\in H$ and $\tilde{\bm{u}} = P_H(P_{\mathcal{K}}(\bm{u}))$. Moreover, we write $\bm{u}$ and $\tilde{\bm{u}}$ as
$
\bm{u} = \bm{u}_*+\bm{B}\bm{t}$ and $\tilde{\bm{u}} = \bm{u}_*+\bm{B}\tilde{\bm{t}}$, respectively. Assume that $\bm{u}\neq\bm{u}_*$ and $\bm{u}, \tilde{\bm{u}}\not\in\mathcal{K}$. Then, we have
\begin{align}
    \label{update}
    \tilde{\bm{t}} &=\bm{B}^T(\lambda_1\bm{e}_1 -\bm{u}_*),
\end{align}
where $\lambda_1(\bm{t})$ and $\bm{e}_1$ are the coefficients and associated vector of the first component of the spectral decomposition of $\bm{u}$. 
\end{lemma}

We analyze this case using the APM update rule \eqref{update}. For this, we use $\bm{b}\bm{b}^T+\bm{B}_1^T\bm{B}_1=\bm{I}$ and $\bm{B}_1\bm{u}_{*1} = u_{*0}\bm{b}$. For simplicity, we introduce some notation and symbols as follows:
\begin{align*}
X &= u_{*0} +\bm{b}^T\bm{t}, Y = \|\bm{B}_1\bm{t}\|^2 -(\bm{b}^T\bm{t})^2=\|\bm{t}\|^2 -2(\bm{b}^T\bm{t})^2,\\
P &= \|\bm{u}_{*1} +\bm{B}_1\bm{t}\| = \sqrt{u_{*0}^2+2u_{*0}\bm{b}^T\bm{t} + \|\bm{B}_1\bm{t}\|^2} = \sqrt{X^2+Y}.
\end{align*}
We remark that $P-X = Y/(P+X)$. Then from \eqref{update}, we obtain 
\begin{align*}
\tilde{\bm{t}}&=\frac{\lambda_1}{\sqrt{2}}\left(\frac{1}{\|\bm{u}_{*1}+\bm{B}_1\bm{t}\|}\bm{t} + \frac{u_{*0}-\bm{b}^T\bm{t}+\|\bm{u}_{*1}+\bm{B}_1\bm{t}\|}{\|\bm{u}_{*1}+\bm{B}_1\bm{t}\|}\bm{b}\right)-2u_{*0}\bm{b}\\
&=\frac{1}{2}\left(\frac{u_{*0}^2-(\bm{b}^T\bm{t})^2 -2u_{*0} \|\bm{u}_{*1}+\bm{B}_1\bm{t}\| + \|\bm{u}_{*1}+\bm{B}_1\bm{t}\| ^2}{\|\bm{u}_{*1} + \bm{B}_1\bm{t}\|}\right)\bm{b} + \frac{1}{2}\left(\frac{u_{*0}+\bm{b}^T\bm{t}}{\|\bm{u}_{*1}+\bm{B}_1\bm{t}\|}+1\right)\bm{t}\\
&= -\frac{1}{2}\left(
    \frac{X}{P}-1
    \right)
    \left(P-X+2\bm{b}^T\bm{t}\right)\bm{b}+ \frac{1}{2}\left(\frac{X}{P}-1\right)\bm{t}+\bm{t}\\
&= \bm{t} -\frac{1}{2}\left(1-\frac{X}{P}\right)(\bm{t}-2(\bm{b}^T\bm{t})\bm{b}) + \frac{1}{2}\left(1-\frac{X}{P}\right)(P-X)\bm{b}\\
&=\bm{t} - \frac{Y}{2P(P+X)}(\bm{t}-2(\bm{b}^T\bm{t})\bm{b}) + \frac{Y^2}{2P(P+X)^2}\bm{b}.
\end{align*}
Here we have $X = u_{*0} + O(\|\bm{t}\|)$, $P = u_{*0} + O(\|\bm{t}\|)$, $Y^2 = O(\|\bm{t}\|^4)$, and 
\begin{align*}
2P(P+X)&= 4u_{*0}^2 + O(\|\bm{t}\|), 2P(P+X)^2 = 8u_{*0}^3 + O(\|\bm{t}||).
\end{align*}
Using these equations, we obtain 
\begin{align*}
    \tilde{\bm{t}} &= \bm{t} - \frac{1}{2}\left(
    \frac{\|\bm{t}\|^2 - 2(\bm{b}^T\bm{t})^2}{2u_{*0}^2}
    \right)\left(\bm{t}-2(\bm{b}^T\bm{t})\bm{b}\right) + O(\|\bm{t}\|^4). 
\end{align*}
Therefore, we obtain the following update rule for this case:
\[
\bm{t}_{k+1} =\bm{t}_k - \frac{1}{2}\left(
    \frac{\|\bm{t}_k\|^2 - 2(\bm{b}^T\bm{t}_k)^2}{2u_{*0}^2}
    \right)\left(\bm{t}_k-2(\bm{b}^T\bm{t}_k)\bm{b}\right) + O(\|\bm{t}_k\|^4) \ (k=0, 1, \ldots).
\]
From this equation, we obtain
\[
\|\bm{t}_{k+1}\|^2 = \|\bm{t}_k\|^2 -\frac{1}{2}\left(
\frac{\|\bm{t}_k\|^2 - 2(\bm{b}^T\bm{t}_k)^2}{u_{*0}}
\right)^2 + O(\|\bm{t}_k\|^5) \ (k=0, 1, \ldots, ).
\]
Using the Taylor expansion of $\sqrt{1-X}$ around $X=0$, we can reduce the equation as follows:
\begin{align}\label{eq3}
\|\bm{t}_{k+1}\|&=\|\bm{t}_k\|\left(
1-\frac{1}{4\|\bm{t}_k\|^2}\left(
\frac{\|\bm{t}_k\|^2 -2(\bm{b}^T\bm{t}_k)^2}{u_{*0}}
\right)^2 + O(\|\bm{t}_k\|^3)
\right). 
\end{align}

We can prove the following theorem by using \eqref{eq3} and Lemma \ref{lemma1}. Item \ref{Conv2} of Theorem \ref{thm:convrate} follows from this theorem. 
\begin{theorem}\label{thm:case2}
The following hold:
\begin{enumerate}[label=(\arabic*)]
\item \label{Q1} If $p=1$, then $\displaystyle\lim_{k\to\infty}\left(\frac{1-2b^2}{\sqrt{2}u_{*0}}\right)^2k^{1/2}|t_k| = 1$. 
\item\label{Q2} If $p\ge 2$, then we have 
\[
\sqrt{2}u_{*0}\le \liminf_{k\to\infty}k^{1/2}\|\bm{t}_k\|, \limsup_{k\to\infty}k^{1/2}\|\bm{t}_k\|\le \frac{\sqrt{2}u_{*0}}{1-2\|\bm{b}\|^2}. 
\]
\end{enumerate}
\end{theorem}
\begin{remark}
In \ref{Q1} of Theorem \ref{thm:case2}, the vectors $\bm{b}$ and $\bm{t}_k$ are not written with the bold font because $p=1$. Moreover, it follows from Theorem~\ref{thm:case2} that the convergence rate of the APM for Case 2 is $\Theta(k^{-1/2})$.  Indeed, it is clear that $|t_k|=\Theta(k^{-1/2})$ when $p=1$. In addition, it follows from the inequalities in \ref{Q2} of Theorem \ref{thm:case2} that for a fixed $\epsilon >0$, there exists $\tilde{k}\in\mathbb{N}$ such that for all $k\ge \tilde{k}$, 
\[
k^{-1/2}(\sqrt{2}u_{*0} -\epsilon) \le \|\bm{t}_k\|\le k^{-1/2}\left(\frac{\sqrt{2}u_{*0}}{1-2\|\bm{b}\|^2} + \epsilon\right).
\]
This implies $\|\bm{t}_k\| = \Theta(k^{-1/2})$. 
\end{remark}

\begin{proof}[Proof of Theorem \ref{thm:case2}] If $p=1$, then  \eqref{eq3} can be rewritten as
\[
|t_{k+1}| = |t_k|\left(1 -\frac{1}{4}\left(\frac{1-2b^2}{u_{*0}}\right)^2|t_k|^2 +O(|t_k|^3)\right)  \ (k=0, 1, \ldots). 
\]
Applying \ref{L2} and \ref{L3} of Lemma \ref{lemma1}, we obtain the desired equation.

It follows from Cauchy-Schwarz inequality that $(1-2\|\bm{b}\|^2)\|\bm{t}\|^2 \le \|\bm{t}\|^2 - 2(\bm{b}^T\bm{t})^2\le \|\bm{t}\|^2$ for all $\bm{t}\in\mathbb{R}^p$. Using these inequalities, we obtain 
\begin{align*}
\|\bm{t}_{k+1}\| &\ge \|\bm{t}_k\|\left(1 - \displaystyle\frac{\|\bm{t}_k\|^2}{4u_{*0}^2} + O(\|\bm{t}_k\|^3)\right), \\
\|\bm{t}_{k+1}\| &\le \|\bm{t}_k\|\left(1 - \displaystyle\frac{1}{4}\left(\frac{1-2\|\bm{b}\|^2}{u_{*0}}\right)^2\|\bm{t}_k\|^2 + O(\|\bm{t}_k\|^3)\right), 
\end{align*}
for all $k=0, 1, \ldots$. Applying \ref{L2} and \ref{L3}  of Lemma \ref{lemma1} to these inequalities, we obtain the desired inequalities. 
\end{proof}

\subsection{Case 3: $H\cap\bd\mathcal{K}=C(\bm{d})$}\label{sec:Case3} Recall that $\|\bm{b}\|^2 =1/2$ from \ref{lemma:P4} of Lemma ~\ref{lemma:socp}. Using this property, we can write $\bm{d}$ as $\bm{d}=\bm{B}\bm{b}/\|\bm{b}\|$. It is clear that $\bm{d}\in\bd\mathcal{K}$. Also, we rewrite the affine subspace $H$ in terms of an orthogonal matrix as follows:
\[
H=\left\{\bm{D}\bm{t} : \bm{t}\in\mathbb{R}^p
\right\}, \bm{D} := \begin{bmatrix}
\bm{d} & \tilde{\bm{D}}
\end{bmatrix}:=\begin{bmatrix}
1/\sqrt{2} & \bm{0}^T\\
\bm{B}_1\bm{b}/\|\bm{b}\| & \bm{D}_1
\end{bmatrix}, 
\]
where the matrix $\bm{D}_1$ satisfies $\bm{D}_1^T\bm{D}_1=\bm{I}_{p-1}$ and $\bm{b}^T\bm{B}_1^T\bm{D}_1=\bm{0}$. 

For $\bm{t}=(x; \bm{y})\in\mathbb{R}^p$, $\lambda_1$ and $\bm{e}_1$ in the spectral decomposition of $\bm{u}=\bm{D}\bm{t}$ are 
\begin{align*}
\lambda_1 &= \frac{1}{\sqrt{2}}\left(
\frac{x}{\sqrt{2}}+\|\bm{z}\|
\right)= \frac{1}{2}\left(
x+\sqrt{x^2 +2\|\bm{y}\|^2}
\right), \\
\bm{e}_1&=\frac{1}{\sqrt{2}}\left(1; \frac{\bm{z}}{\|\bm{z}\|}\right), \mbox{ where } \bm{z} = \frac{\bm{B}_1\bm{b}}{\|\bm{b}\|}x + \bm{D}_1\bm{y}. 
\end{align*}
We see $\|\bm{z}\| = \sqrt{x^2/2 + \|\bm{y}\|^2}$ by using $\bm{b}\bm{b}^T+\bm{B}_1^T\bm{B}_1=\bm{I}$. If $\bm{y}=\bm{0}$ and $x\ge 0$, then $\bm{u}\in C(\bm{d})$. If $\bm{y}=\bm{0}$ and $x<0$, then $P_{\mathcal{K}}(\bm{u}) =\bm{0}\in C(\bm{d})$. These facts mean that if we choose $x\bm{d}$ as the initial point of the APM, then it finds a point in $C(\bm{d})$ in at most one iteration. If $\bm{y}\neq \bm{0}$, we have $P_{\mathcal{K}}(\bm{u})=\lambda_1\bm{e}_1$. Thus, we can write $\tilde{\bm{t}}$ obtained by one iteration of the APM for $\bm{t}$ as follows: 
\[
\tilde{\bm{t}} = \lambda_1 \bm{D}^T\bm{e}_1 = \frac{1}{\sqrt{2}}\left(
\frac{x}{\sqrt{2}}+\|\bm{z}\|
\right)
\begin{bmatrix}
\frac{1}{2}+\frac{x}{2\sqrt{2}\|\bm{z}\|}\\
\frac{\bm{y}}{\sqrt{2}\|\bm{z}\|}
\end{bmatrix} = \frac{1}{2}\begin{bmatrix}
x + \frac{x^2+\|\bm{y}\|^2}{\sqrt{x^2+2\|\bm{y}\|^2}}\\
\bm{y} + \frac{x}{\sqrt{x^2+2\|\bm{y}\|^2}}\bm{y}
\end{bmatrix}. 
\]

We consider the sequence $\{\bm{t}_k\}$ generated by the APM with the initial point $\bm{t}_0=(x_0; \bm{y}_0)$, where $\bm{y}_0\neq \bm{0}$. For simplicity, we write $\bm{t}_k =(x_k; \bm{y}_k)$ for all $k$, and then we obtain the following recurrence formula:
\begin{align}\label{eq:recurrence}
\left\{
\begin{array}{lcl}
x_{k+1} &=&  \displaystyle\frac{1}{2}\left(
x_k + \frac{x_k^2+\|\bm{y}_k\|^2}{\sqrt{x_k^2+2\|\bm{y}_k\|^2}}\right), \\
\bm{y}_{k+1}&=& \displaystyle\frac{1}{2}\left(1 + \frac{x_k}{\sqrt{x_k^2+2\|\bm{y}_k\|^2}}\right)\bm{y}_k, 
\end{array}
\right. \ (k=0, 1, \ldots). 
\end{align}

From the recurrence formula \eqref{eq:recurrence}, we can see that (i) $\{x_k\}$ is monotonically increasing, (ii) $\{\|\bm{y}_k\|\}$ is monotonically decreasing,  (iii) $x_k  > 0$ for all $k=1, 2, \ldots$, and (iv) $\bm{y}_k\neq \bm{0}$ for all $k$ if $\bm{y}_0\neq\bm{0}$. 
Indeed, we can see these properties from the following inequalities: for any $\bm{y}\neq \bm{0}$, $
\left|x\right| < \sqrt{x^2+2\|\bm{y}\|^2}$ and $|x\sqrt{x^2+2\|\bm{y}\|^2}| < x^2+\|\bm{y}\|^2$.

We find the convergence rate of the APM. The convergence rate can be summarized in the following theorem. Item \ref{Conv3} of Theorem~\ref{thm:convrate} follows from this theorem.

\begin{theorem}\label{thm:case3}
The following hold for \eqref{eq:recurrence}:
\begin{enumerate}[label=(\arabic*)]
\item\label{Case3-1} If $\bm{y}_0 = \bm{0}$, then $(x_0, \bm{y}_0)\in H\cap\mathcal{K}$ or  $(x_1, \bm{y}_1)\in H\cap\mathcal{K}$. 
\item\label{Case3-2} If $\bm{y}_0\neq \bm{0}$, then the convergence rate of the APM is $\Theta(k^{-1/2})$. 
\end{enumerate}
\end{theorem}
\begin{proof}
\ref{Case3-1} : If $\bm{y}_0 =\bm{0}$, then $\bm{y}_1=\bm{0}$ and $x_1 = (x_0 + |x_0|)/2$. Thus, $(x_0, \bm{y}_0)\in H\cap\mathcal{K}$ or $(x_1, \bm{y}_1)\in H\cap\mathcal{K}$. 

\ref{Case3-2} : It follows from the property (iv) that the APM does not terminate in finitely many iterations. Furthermore, since the singularity degree of this case is $1$, it follows from \cite[Theorem 2.4]{Drusvyatskiy2017} that the sequence $\{(x_k; \bm{y}_k)\}$ converges to a point $(x_\infty; \bm{0})\in\mathbb{R}^p$ with $O(k^{-1/2})$ rate, where $x_\infty\ge 0$. In particular, from the property (i), $x_\infty$ should be positive. 

We find the convergence rate of the sequence $\{\|\bm{y}_k\|\}$. For this, we define $f(x, y)$ as $f(x, y) = \frac{\sqrt{x^2 +2y^2}-x}{2y^2\sqrt{x^2+2y^2}}$. 
Then we can make an equation on $\|\bm{y}_k\|$ from \eqref{eq:recurrence} as follows:
\begin{align*}
\|\bm{y}_{k+1}\| &=\frac{\|\bm{y}_k\|}{2}\left(1 +\frac{x_k}{\sqrt{x_k^2 + 2\|\bm{y}_k\|^2}}\right) = \|\bm{y}_k\|\left(1 - \frac{1}{2}\left(\frac{\sqrt{x_k^2+2\|\bm{y}_k\|^2}-x_k}{\sqrt{x_k^2 + 2\|\bm{y}_k\|^2}}\right)\right)\\
&=\|\bm{y}_k\|\left(1 - \|\bm{y}_k\|^2 f(x_k, \|\bm{y}_k\|)\right) \ (k=0, 1, \ldots). 
\end{align*}
It follows from \ref{L4} of Lemma~\ref{lemma1} that $\lim_{k\to\infty}1/(k\|\bm{y}_k\|^2) = 2f(x_\infty, 0) = 1/x_\infty^2$, and thus $\|\bm{y}_k\|=\Theta(k^{-1/2})$. 
%
%

We consider the convergence rate of the sequence $\{x_k\}$. For this, we define a sequence $\{a_k\}\subset\mathbb{R}$ with   
\[
a_k := \frac{k^2\left(\sqrt{x_k^2+2\|\bm{y}_k\|^2} -x_k\right)^2}{2\sqrt{x_k^2+2\|\bm{y}_k\|^2}} = \frac{2k^2\|\bm{y}_k\|^4}{\sqrt{x_k^2 + 2\|\bm{y}_k\|^2}(\sqrt{x_k^2+2\|\bm{y}_k\|^2} +x_k)^2}\ (k=0, 1, \ldots)
\]
We note that $\displaystyle\lim_{k\to\infty}a_k = x_\infty/2$. Then we have 
\begin{align*}
x_{k+1}-x_k &= \frac{1}{2}\left(
-x_k + \frac{x_k^2 + \|\bm{y}_{k}\|^2}{\sqrt{x_k^2 + 2\|\bm{y}_{k}\|^2}}
\right) = \frac{\left(\sqrt{x_k^2 + 2\|\bm{y}_{k}\|^2} - x_k\right)^2}{2\sqrt{x_k^2 + 2\|\bm{y}_k\|^2}} = \frac{a_k}{k^2}. 
\end{align*}
It follows from \ref{L5} of Lemma~\ref{lemma1} that $\displaystyle\lim_{k\to\infty}k(x_\infty - x_k) =x_\infty/2$, which implies $x_\infty-x_k=\Theta(1/k)$. Therefore, $\|\bm{t}_k-(x_\infty; \bm{0})\|=\|(x_k; \bm{y}_k)-(x_\infty; \bm{0})\|=\Theta(k^{-1/2})$.
\end{proof}

\begin{remark}
In the case where $H\cap\mathcal{K}=C(\bm{d})$ for some  $\bm{d}\in\bd\mathcal{K}\setminus\{\bm{0}\}$, the linear regularity may not hold. Indeed, we can show that the linear regularity fails in the second-order cone and the affine subspace defined by the following $\bm{u}_*$ and $\bm{B}$: 
\[
\bm{u}_*=\bm{0}, \bm{B} = \begin{bmatrix}
1/\sqrt{2} & 0\\
1/\sqrt{2} & 0\\
0 & 1
\end{bmatrix}. 
\]
Instead, the $1/2$-H{\"o}lder regularity holds. Thus, it follows from \cite[Proposition 4.2]{Borwein2014} and \cite[Theorem 2.2]{Drusvyatskiy2017} that the worst-case convergence rate of the APM is $O(k^{-1/2})$. Our result gives, in addition to this, a lower bound on the convergence rate. 
\end{remark}

\section{Examples of the intersection of an affine subspace and the product set of two second-order cones}\label{sec:product}

We consider the convergence rate of the APM for the intersection of an affine subspace and the product set of two second-order cones. In this case, since the singularity degree can be two, it follows from \cite[Theorem 6.7]{Lourenco2018} that the worst-case convergence rate of the APM is $O(k^{-1/6})$\footnote{More precisely, an error bound is provided by using the value called the distance to the partial polyhedral Slater's condition in \cite[Theorem 6.7]{Lourenco2018}. From the definition, the value is smaller than or equal to the singularity degree. Thus, the error bound with this value may give a better convergence rate of the APM. However, in all examples of this section, the value coincides with the singularity degree. Therefore, we can consider the convergence rate derived from singularity degree}. 

In this section, we introduce two examples. Since the singularity degrees of both examples are two, it is expected that the worst-case convergence rate is $O(k^{-1/6})$. However,  although the worst-case convergence rate $O(k^{-1/6})$ in the first example is tight, the worst-case convergence rate $O(k^{-1/6})$ in the second case is {\itshape not} tight. 
In other words, the convergence rate of the APM for any initial guess in the second example is faster than the previously known worst-case convergence rate $O(k^{-1/6})$. 

We briefly introduce the APM for the intersection of an affine subspace $H$ and the product set $K$ of two second-order cones. Let $K = \mathcal{K}_{m+1}\times \mathcal{K}_{n+1}$. We write the affine subspace $H\subset\mathbb{R}^{(m+1)+(n+1)}$ by
\[
H = \left\{
\begin{bmatrix}
    \bm{u}_* \\
    \bm{v}_* 
\end{bmatrix}+ \begin{bmatrix}
    \bm{B}\\
    \bm{C}
\end{bmatrix}\bm{t} : \bm{t}\in\mathbb{R}^p
\right\}. 
\]
Here, we assume that each column vector of the matrix $\begin{bmatrix}\bm{B}\\ \bm{C}\end{bmatrix}$ is orthonormal.  In addition, for a given $\bm{t}\in\mathbb{R}^p$, we define $\bm{u}(\bm{t})$ and $\bm{v}(\bm{t})$ by
\[
\bm{u}(\bm{t}) = \bm{u}_*+\bm{B}\bm{t} \mbox{ and } \bm{v}(\bm{t}) = \bm{v}_*+\bm{C}\bm{t}. 
\]
We describe one iteration of the APM. For this, we assume that the APM generates $(\bm{u}(\tilde{\bm{t}}); \bm{v}(\tilde{\bm{t}})) \in H $ from $(\bm{u}(\bm{t}); \bm{v}(\bm{t}))\in H$ with one  iteration. Then, the projection $P_K((\bm{u}(\bm{t}); \bm{v}(\bm{t}))$ of $(\bm{u}(\bm{t}); \bm{v}(\bm{t}))$ into $K$ is 
\[
P_K((\bm{u}(\bm{t}); \bm{v}(\bm{t}))) = (P_{\mathcal{K}}(\bm{u}(\bm{t}));P_{\mathcal{K}}(\bm{v}(\bm{t}))).
\]
Recall that $P_{\mathcal{K}}(\bm{u})$ is the projection of $\bm{u}$ into the second-order cone $\mathcal{K}$. Thus by applying the projection $P_H$ to the affine subspace $H$ for $P_K((\bm{u}(\bm{t}); \bm{v}(\bm{t})))$, we obtain 
\[
\tilde{\bm{t}} = \begin{bmatrix}
    \bm{B}^T & \bm{C}^T
\end{bmatrix}\begin{bmatrix}
    P_{\mathcal{K}}(\bm{u}(\bm{t}))-\bm{u}_*\\
    P_{\mathcal{K}}(\bm{v}(\bm{t}))-\bm{v}_*
\end{bmatrix}. 
\]
This is the update of the parameter $\bm{t}$ which spends one iteration of the APM. For a sequence $\{(\bm{u}_k; \bm{v}_k)\}$ of the APM, we can consider the sequence of parameters $\{\bm{t}_k\}_k$ of $\{(\bm{u}_k; \bm{v}_k)\}$. Then, as we have already discussed in Section~\ref{sec:convergence}, the convergence rate of $\{\bm{t}_k\}$ is the same as that of $\{(\bm{u}_k; \bm{v}_k)\}$. Therefore, we consider the convergence rate of $\{\bm{t}_k\}$ in the following examples. 

%
%
In the following examples, we deal with $K=\mathcal{K}_3\times \mathcal{K}_3$

\subsection{Example 1: the convergence rate $O(k^{-1/6})$ is tight}\label{example1}
We define the affine subspace $H$ by 
\[
\bm{u}_* = \bm{v}_* =\begin{bmatrix}
1\\
1\\
0
\end{bmatrix}, \bm{B} = \begin{bmatrix}
    1/\sqrt{3} & 0\\
    1/\sqrt{3}  & 0\\
    0 & 1/\sqrt{6} 
\end{bmatrix}, \bm{C} = \begin{bmatrix}
  0 & 2/\sqrt{6}\\
  0 & 1/\sqrt{6}\\
  1/\sqrt{3} & 0
\end{bmatrix}. 
\]
Then we can prove that $H\cap K = \{(\bm{u}_*; \bm{v}_*)\}$ and that the singularity degree is two. In fact, the minimal face of $K$ is $C(\bm{u}_*)\times C(\bm{v}_*)$. Facial reduction finds this face by using the reducing certificates
$\bm{d}_1:=(1, -1, 0, 0, 0, 0)^T$, $\bm{d}_2:=(1, -1, -1, 1, -1, 0)^T$, which corresponds to $\bm{y}$ in Theorem~\ref{thm:fr}. Moreover, it is easy to prove that there does not exist any reducing certificate $\bm{y}$ such that $K\cap\{\bm{y}\}^\perp =C(\bm{u}_*)\times C(\bm{v}_*)$, which implies the singularity degree is not one. 

The following theorem implies that the convergence rate of the APM for this example coincides with the worst-case  convergence rate induced by the error bound for this example. We give a proof of this theorem in Appendix \ref{app:convrateEx1}.

\begin{theorem}\label{thm:example1}
There exists an initial guess $(\bm{u}_0;\bm{v}_0)\in H$ such that the convergence rate of the generated sequence $\{(\bm{u}_k; \bm{v}_k)\}\subset H$ is $\Theta(k^{-1/6}
)$.
\end{theorem}

\subsection{Example 2: the convergence rate $O(k^{-1/6})$ is NOT tight}\label{example2}
We define the affine subspace $H$ by 
\[
\bm{u}_* = \bm{v}_*=\bm{0}, \bm{B} = \begin{bmatrix}
    a & 0 &0\\
    a & 0 &0\\
    0 & 0 & 2b
\end{bmatrix}, \bm{C} = \begin{bmatrix}
  0 & a & b\\
  0 & a & -b\\
  0& 0& 2b
\end{bmatrix}, 
\]
where $a=1/\sqrt{2}$ and  $b=1/\sqrt{10}$. 
Then we can prove that $H\cap K = C(\bm{d})\times C(\bm{d})$, where $\bm{d} = (1, 1, 0)^T$. In addition, we can prove that the singularity degree is two. Indeed, the first reducing certificate is $\bm{d}_1:=(1, -1, 0, 0, 0, 0)^T$ and the second one is $\bm{d}_2 := (1, -1, -1, 1, -1, 0)^T$. Moreover, there does not exist any reducing certificate $\bm{y}\neq \bm{0}$ such that $K\cap\{\bm{y}\}^\perp = C(\bm{d})\times C(\bm{d})$, which is the minimal face of $H\cap K$. 

For $\bm{t}=(x, y, z)$, $(\bm{u}(\bm{t}); \bm{v}(\bm{t}))\in H$ is the form of 
\[
\bm{u}(\bm{t}) = \begin{bmatrix}
ax\\
ax\\
2bz
\end{bmatrix}, \bm{v}(\bm{t}) = \begin{bmatrix}
ay+bz\\
ay-bz\\
2bz
\end{bmatrix}.
\]
For simplicity, we define $X = ax$, $Y = ay$ and $Z=bz$ and consider the sequence $\{\bm{t}_k\} := \{(X_k, Y_k, Z_k)\}$ generated by the APM for this example. The following theorem implies that the worst-case convergence rate of the APM is better than the worst-case convergence rate induced by the error bound with the singularity degree of this example.  We will give the proof of this theorem in Appendix~\ref{app:property}. 
\begin{theorem}\label{thm:example2}
Let $(X_\infty, Y_\infty, Z_\infty)$ be the convergence point of the sequence generated by the APM. Then, the following hold. 
\begin{enumerate}[label=(C\arabic*)]
\item\label{C0} If $Z_0=0$, then $(X_0, Y_0, Z_0)\in H\cap K$ or $(X_1, Y_1, Z_1)\in H\cap K$.  
\item\label{C1} If $Z_0 > 0$, then we have $X_k-X_\infty = \Theta(k^{-1})$, $Y_k-Y_\infty = \Theta(k^{-1/2})$ and $|Z_k-Z_\infty|=\Theta(k^{-1/2})$. 
\item\label{C2} If $Z_0 < 0$, then we have $X_k-X_\infty = \Theta((4/5)^{4k})$, $Y_\infty-Y_k = \Theta((4/5)^{3k})$ and $|Z_k-Z_\infty|=\Theta((4/5)^k)$. 
\end{enumerate}
\end{theorem}

\section*{Acknowledgement}
The first author was supported by JSPS KAKENHI Grant Number JP17K18726 and JSPS Grant-in-Aid for Transformative Research Areas (A) (22H05107). The second author was supported by JSPS KAKENHI Grant Number JP19K03631. The third author was supported by JSPS KAKENHI Grant Number JP20K11696 and ERATO HASUO Metamathematics for Systems Design Project (No.JPMJER1603), JST.

\appendix 
\section{Proofs of Lemmas~\ref{lemma:socp0}, \ref{lemma:matrix} and  \ref{lemma:socp}}\label{sec:proofoflemma}
\subsection{Proof of Lemma~\ref{lemma:socp0}}
Since $\bm{x}\in\bd\mathcal{K}$, $x_0=\|\bm{x}_1\|$. 
In addition, it follows from $\bm{x}^T\bm{y}=0$ and the Cauchy-Schwartz inequality that $y_0=\|\bm{y}_1\|$. Substituting $x_0=\|\bm{x}_1\|$ and $y_0=\|\bm{y}_1\|$ into $\bm{x}^T\bm{y}$, we obtain $\|\bm{x}_1\|\|\bm{y}_1\| = -\bm{x}_1^T\bm{y}_1$, and can conclude $\bm{y}_1 = -\alpha \bm{x}_1$ for some $\alpha \in\mathbb{R}$. In addition, we have $y_0 =\|\bm{y}_1\|=|\alpha|\|\bm{x}_1\|=|\alpha|x_0$. If $\alpha < 0$, then $\bm{y} = |\alpha|\bm{x}$, and this contradicts $\bm{x}^T\bm{y}=0$ and $\bm{x}\neq \bm{0}$. Therefore, $\alpha \ge 0$ and we can conclude $\bm{y}= \alpha \hat{\bm{x}}$ for some $\alpha\ge 0$. 

\subsection{Proof of Lemma~\ref{lemma:matrix}}
Let $\bm{C}:=\bm{I}_p-2\bm{b}\bm{b}^T$.  For a vector $\bm{q}$ orthogonal to $\bm{b}$, $\bm{C}\bm{q} = \bm{q}$. Since we can construct $(p-1)$ linearly independent vectors orthogonal to $\bm{b}$, the (algebraic) multiplicity with eigenvalue 1 is $(p-1)$. Also, since $\bm{C}\bm{b} = (1-2\|\bm{b}\|^2)\bm{b}$, the eigenvector is $\bm{b}$ with eigenvalue $1-2\|\bm{b}\|^2$. Since the size of the matrix $\bm{C}$ is $p$, the (algebraic) multiplicity is $1$.   

\subsection{Proof of \ref{lemma:P2} of Lemma~\ref{lemma:socp}}
$\bm{u}_*+\bm{B}\bm{t}\in\mathcal{K}$ is equivalent to $u_{*0} + \bm{b}^T\bm{t} \ge 0$ and 
\[
 (u_{*0}+\bm{b}^T\bm{t})^2 - \|\bm{u}_{*1} + \bm{B}_1\bm{t}\|^2=2(\bm{b}u_{*0}-\bm{B}_1^T\bm{u}_{*1})^T\bm{t} + \bm{t}
^T\left(2\bm{b}\bm{b}^T-\bm{I}_p\right)\bm{t} \ge 0. 
\]
Here we used $\bm{B}^T\bm{B}=\bm{b}\bm{b}^T + \bm{B}_1^T\bm{B}_1=\bm{I}_p$.
In the following, we consider two cases, $u_{*0}=0$ and $u_{*0}\neq 0$. 

First, if $u_{*0}=0$, then $\bm{u}_*=\bm{0}$. Therefore, $\bm{B}\bm{t}\in\mathcal{K}$ is equivalent to 
\[
\bm{b}^T\bm{t} \ge 0, \bm{t}^T\left(2\bm{b}\bm{b}^T-\bm{I}_p\right)\bm{t}\ge 0. 
\]
Hence, $H\cap\bd\mathcal{K}=\{\bm{0}\}$ if and only if for any $\bm{t}\neq \bm{0}$ with $\bm{b}^T\bm{t}\ge 0$, 
$\bm{t}^T\left(2\bm{b}\bm{b}^T-\bm{I}_p\right)\bm{t} < 0$, 
Substituting $\bm{t}=\bm{b}$ into the last inequality, we obtain $2\|\bm{b}\|^2 < 1$. Furthermore, if $2\|\bm{b}\|^2 < 1$, then it follows from Lemma ~\ref{lemma:matrix} that  the matrix $2\bm{b}\bm{b}^T-\bm{I}_p$ is negative definite. Thus, under the assumption that $u_{*0}=0$, $H\cap\bd\mathcal{K}=\{\bm{0}\}$ if and only if $\|\bm{b}\|^2 < 1/2$. 

Next, we consider the case $u_{*0}\neq 0$. Since $u_{*0}>0$, the inequality $u_{*0}-\bm{b}^T\bm{t} \ge 0$ holds for $\|\bm{t}\|$ sufficiently close to 0. Thus, $H\cap \bd\mathcal{K}=\{\bm{u}_*\}$ if and only if for all $\bm{t}\in\mathbb{R}^p$, 
\[
2(\bm{b}u_{*0}-\bm{B}_1^T\bm{u}_{*1})^T\bm{t} + \bm{t}
^T\left(2\bm{b}\bm{b}^T-\bm{I}_p\right)\bm{t} \le  0. 
\]
If $\bm{b}u_{*0}-\bm{B}_1^T\bm{u}_{*1}\neq \bm{0}$, then we can construct $\bm{t}\neq \bm{0}$ that does not satisfy the above inequality. Thus $\bm{B}^T\bm{u}_*=\bm{b}u_{*0}-\bm{B}_1^T\bm{u}_{*1}= \bm{0}$ holds. Furthermore, the matrix $(2\bm{b}\bm{b}^T-\bm{I}_p)$ must be negative definite. 

Therefore, we can conclude that $H\cap \bd\mathcal{K}=\{\bm{u}_*\}$ if and only if $\bm{B}^T\bm{u}_*=\bm{0}$ and $2\|\bm{b}\|^2 < 1$. 

\subsection{Proof of \ref{lemma:P4} of Lemma~\ref{lemma:socp}}
Assume that $H\cap\mathcal{K}=\{\alpha \bm{d} : \alpha \ge 0\}$ for some $\bm{d}\in\bd\mathcal{K}\setminus\{\bm{0}\}$. Then there exists $\bm{t}\neq \bm{0}$ such that $\bm{B}\bm{t}\in\bd\mathcal{K}\setminus\{\bm{0}\}$.  Using $\bm{b}\bm{b}^T+\bm{B}_1^T\bm{B}_1=\bm{I}$, we obtain
\[
\bm{t}^T\left(\bm{I}-2\bm{b}\bm{b}^T\right)\bm{t} = 0. 
\]
If $\|\bm{b}\|^2<1/2$, then the matrix $\bm{I}-2\bm{b}\bm{b}^T$ is positive definite, and thus $\|\bm{b}\|^2\ge 1/2$ holds. 

Assume that $\|\bm{b}\|^2 > 1/2$. Then, it follows from  $\|\bm{B}_1\bm{b}\|=\|\bm{b}\|\sqrt{1-\|\bm{b}\|^2}$ that the inequality $\|\bm{b}\|^2 > \|\bm{B}_1\bm{b}\|$ holds. Thus $\bm{B}\bm{b}\in H \cap \inte \mathcal{K}$, which contradicts $H\cap\inte\mathcal{K}=\emptyset$. Therefore $\|\bm{b}\|^2 = 1/2$ holds. 

We assume that $\|\bm{b}\|^2 = 1/2$. Then from $\bm{b}\bm{b}^T + \bm{B}_1^T\bm{B}_1 = \bm{I}$, we obtain $\|\bm{B}\bm{b}\| = 1/2$. Thus, choosing $\bm{d} = \bm{B}\bm{b}/\|\bm{b}\|$, we obtain
$\{\alpha \bm{d} : \alpha\ge 0\} \subset H\cap\mathcal{K}$. Suppose that there exists $\bm{f}\in H\cap\bd\mathcal{K}\setminus\{\bm{0}\}$ such that $\bm{f}$ is not parallel to $\bm{d}$. Then $\bm{d}+\bm{f} \in H \cap \inte\mathcal{K}$, and this contradicts that $H\cap \inte\mathcal{K}=\emptyset$. Therefore, we obtain $\{\alpha \bm{d} : \alpha\ge \} = H\cap\mathcal{K}$.

\subsection{Proof of \ref{lemma:P3} of Lemma~\ref{lemma:socp}}
Define $\bm{u}=(u_0; \bm{u}_1)\in H\setminus\{\bm{u}_*\}$. 
If $-\hat{\bm{u}}\in\mathcal{K}$, then $-u_{*0}-\bm{b}^T\bm{t}\ge 0$ and $(-u_{*0}-\bm{b}^T\bm{t})^2-\|\bm{u}_{*1}+\bm{B}_1\bm{t}\|^2\ge 0$. However, since $\bm{u}_*\in\bd\mathcal{K}$, $\bm{B}\hat{\bm{u}}_*=\bm{0}$ and $\bm{I}_p-2\bm{b}\bm{b}^T\in\mathbb{S}^p_{++}$ from \ref{lemma:P2}, $\bm{t}=\bm{0}$ must hold. This contradicts $\bm{u}\neq\bm{u}_*$.  

\subsection{Proof of \ref{lemma:P6} of Lemma~\ref{lemma:socp}}
We assume $P_{\mathcal{K}}(\bm{u}) = \bm{u}_*$. We define $\bm{u} = \bm{u}_*+\bm{B}\bm{t}$ for some $\bm{t}\neq \bm{0}$.  It follows from Theorem~\ref{thm:pprojsocp} that $P_{\mathcal{K}}(\bm{u}) = \lambda_1\bm{e}_1$, where $\lambda_1$ is the first spectral of $\bm{u}$ and $\bm{e}_1$ is the eigenvector associated with $\lambda_1$. Thus we obtain $\lambda_1\bm{e}_1 = \bm{u}_*$. Then we have
\[
\left\{
\begin{array}{lcl}
u_{*0} -\bm{b}^T\bm{t}&=& \|\bm{u}_{*1} + \bm{B}_1\bm{t}\|,\\
\bm{0}&=& (u_0+\bm{b}^T\bm{t}+\|\bm{u}_{*1} + \bm{B}_1\bm{t}\|)\bm{u}_{*1} + (u_{*0} + \bm{b}^T\bm{t} - \|\bm{u}_{*1} + \bm{B}_1\bm{t}\|)\bm{B}_1\bm{t}.\\
\end{array}
\right.
\]
It follows from these equations that $\bm{0} = (\bm{b}^T\bm{t}) \bm{u}_{*1} + u_0\bm{B}_1\bm{t}$. Then we obtain $\bm{t}=\bm{0}$ by multiplying $\bm{B}_1^T$ from the left and using $\bm{B}_1^T\bm{u}_{*1} = u_{*0}\bm{b}$ and $\bm{b}\bm{b}^T + \bm{B}_1^T\bm{B} = \bm{I}$. This  contradicts $\bm{t}\neq \bm{0}$. 

\subsection{Proof of \ref{lemma:P5} of Lemma~\ref{lemma:socp}}
We remark that $\bm{u}_*=\bm{0}$. Define $\bm{u} = \bm{B}\bm{t}_0$ for some $\bm{t}_0\in\mathbb{R}^p\setminus\{\bm{0}\}$. We assume that $-\hat{\bm{u}}\in\mathcal{K}$. Then we have $-\bm{b}^T\bm{t}_0\ge 0$ and $-\bm{b}^T\bm{t}_0\ge \|\bm{B}_1\bm{t}_0\|$. These imply that $\bm{v}:=-\bm{B}\bm{t}_0\in\mathcal{K}$. Thus, we can write $\bm{v}=\alpha\bm{d}$ for some $\alpha >0$. Consequently, we obtain $\bm{u}=-\alpha \bm{d}$, and which contradicts $\bm{u}\not \in \{\alpha \bm{d} : \alpha \in\mathbb{R}\}$. 

\section{Proof of Proposition~\ref{prop:fr}}\label{sec:proofOfsd}
The faces of the second-order cone are $\{\bm{0}\}$, $C(\bm{d})$ and $\mathcal{K}$, where $\bm{d}\in\bd\mathcal{K}\setminus\{\bm{0}\}$. Among these faces,  we have  a relation $\{\bm{0}\}\subset C(\bm{d})\subset \mathcal{K}$. Furthermore, when $H\cap\mathcal{K}=\{\bm{u}_*\}$, we can find $\bm{y}=\hat{\bm{u}}_*$  in one iteration of facial reduction, and it generates $C(\bm{u}_*)$ as a face of $\mathcal{K}$. Then $H\cap\mathrm{ri}C(\bm{u}_*)\neq\emptyset$ and this means that $C(\bm{u}_*)$ is the minimal face of $\mathcal{K}$. 

Similarly, when  $H\cap\mathcal{K}=C(\bm{d})$, it follows from \ref{lemma:P4} of Lemma~\ref{lemma:socp} that we can choose $\bm{d} = \bm{B}\bm{b}$ and facial reduction generates $C(\bm{d})$ as a face of $\mathcal{K}$, and then $H\cap\mathrm{ri}C(\bm{d})\neq\emptyset$. Therefore, the singularity degree can be greater than 1 only when the intersection of the affine subspace $H$ and the second-order cone $\mathcal{K}$ is at the origin only. That is, the minimal face of $\mathcal{K}$ for $H\cap\mathcal{K}$ is $\{\bm{0}\}$. Therefore, it is enough to prove that the singularity degree is one when $\bm{u}_*=\bm{0}$ and $\|\bm{b}\|^2 < 1/2$. 

We find $\bm{y}\in\inte\mathcal{K}$ that satisfies $\bm{B}^T\bm{y}=\bm{0}$. Indeed, since  $\mathcal{K}^* = \mathcal{K}$ and $\mathcal{K}^\perp = \{\bm{0}\}$, this vector $\bm{y}$ corresponds to $\bm{y}$ in Theorem~\ref{thm:fr}. Moreover, the equality $\mathcal{K}\cap\{\bm{y}\}^\perp =\{\bm{0}\}$ holds, which implies that the singularity degree is one. 
    
    Define the matrix $\bm{C}\in\mathbb{R}^{(n+1)\times (n+1-p)}$ that satisfies $\bm{C}^T\bm{C} = \bm{I}_{n+1-p}$ and $\bm{B}^T\bm{C} = \bm{O}_{p\times (n+1-p)}$. In other words, the matrix $\begin{bmatrix}\bm{B} & \bm{C}\end{bmatrix}$ is orthonormal. In addition, let $\bm{c}$ be the first row of the matrix $\bm{C}$. We partition the matrix $\bm{C}$ as follows:
    \[
\bm{C} = \begin{bmatrix}
\bm{c}^T\\
\bm{C}_1
\end{bmatrix}.
    \]
    Then $\bm{y}:=\bm{C}\bm{c}$ satisfies  $\bm{B}^T\bm{y} = \bm{0}$.  Also, since we have $\|\bm{b}\|^2 + \|\bm{c}\|^2 = 1$, we obtain $\|\bm{c}\|^2 >1/2$. Thus by using $\bm{C}^T\bm{C}=\bm{I}_{n+1-p}$, we have 
    \[
    y_0^2 -\|\bm{y}_1\|^2 = \|\bm{c}\|^4 -\|\bm{C}_1\bm{c}\|^2 = \|\bm{c}\|^2(2\|\bm{c}\|^2 -1) >0, 
    \]
    which implies $\bm{y}=(y_0; \bm{y}_1)\in \inte\mathcal{K}$. 

\section{Proof of Lemma~\ref{lemma1}}\label{sec:proofOfLemma31}
\ref{L2} and \ref{L3} : Since this lemma can be proved in a similar manner of \cite[Corollary 3.2]{Ochiai2021}, we give a sketch of the proof.  

Define $g(x) = (1-(1-Cx^q +x^{q+1}h(x))^q)/(qCx^q)$. Then $g(0) = 1$ and 
\[
qCx_{k+1}^q \le qCx_k^q(1-Cx_k^q + x_{k}^{q+1}h(x_k))^q = qCx_k^q(1-qCx_k^qg(x_k)) \ (k\ge \tilde{k}). 
\]
Since $x_k>0$ and $x_k\to 0$ \ $(k\to\infty)$, there exists $\tilde{k}_0\in\mathbb{N}$ such that $1-qCx_k^qg(x_k)$ is positive for all $k\ge \tilde{k}_0$. 
Thus we have 
\[
\frac{1}{qCx_{k+1}^q}-\frac{1}{qCx_k^q}\ge\frac{g(x_k)}{1-Cx_k^qg(x_k)} 
\]
for all $k\ge \tilde{k}_1:=\max\{\tilde{k}, \tilde{k}_0\}$, which implies
\begin{align*}
\frac{1}{qCkx_k^q} - \frac{1}{qCkx_{K_1}^q} +\frac{1}{k}\sum_{\ell = 0}^{K_1-1} \frac{g(x_{\ell})}{1-Cx_{\ell}^qg(x_{\ell})}  \ge \frac{1}{k}\sum_{\ell = 0}^{k-1} \frac{g(x_{\ell})}{1-Cx_{\ell}^qg(x_k)}. 
\end{align*}
Since the summation in the right-hand side is a Ces{\'a}ro mean and $x_k\to 0$ \ $(k\to\infty)$, we obtain the inequality $\displaystyle\liminf_{k\to\infty}1/(qCkx_k^q)\ge 1$. Therefore, the inequality \eqref{ineqL2} holds. Similarly, we obtain the inequality \eqref{ineqL3}.

\ref{L4} : Define $g(x, w) = ((1-w^qf(x, w))^q-1)/w^q$. Then $g(x_\infty, 0) = -qf(x_\infty, 0)$ and 
\[
\frac{1}{w_{k+1}^q} - \frac{1}{w_k^q} = -\frac{g(x_k, w_k)}{(1-w_k^qf(x_k, w_k))^q} \ (k=0, 1, \ldots)
\]
Thus, we have 
\[
\frac{1}{kw_k^q} - \frac{1}{kw_0^q} = -\frac{1}{k} \sum_{\ell=0}^{k-1}\frac{g(x_\ell, w_\ell)}{(1-w_\ell^q f(x_\ell, w_\ell))^q}. 
\]
Since the summation in the right-hand side is a Ces{\'a}ro mean and $x_k\to x_\infty$ and $w_k\to 0$ \ $(k\to\infty)$, we obtain $\displaystyle\lim_{k\to\infty}1/(kw_k^{q}) = q f(x_\infty, 0)$. 

\ref{L5} : We have $x_\infty - x_k = \displaystyle\sum_{\ell = k}^{\infty} \frac{a_\ell}{\ell^2}$. It should be noted that $\displaystyle\frac{1}{k}\le \sum_{\ell=k}^\infty\frac{1}{\ell^2}\le \frac{1}{k-1}$. Thus, we obtain
\begin{align*}
\frac{\inf\{a_i : i\ge k\}}{k}&\le\sum_{\ell=k}^\infty \frac{\inf\{a_i : i\ge k\}}{\ell^2}\le x_\infty -x_k, \\
x_\infty-x_k&\le \sum_{\ell=k}^\infty \frac{\sup\{a_i : i\ge k\}}{\ell^2}\le \frac{\sup\{a_i : i\ge k\}}{k-1}.
\end{align*}
Hence these inequalities imply
\[
\inf\{a_i : i\ge k\}\le k(x_\infty - x_k)\le \sup\{a_i : i\ge k\}\left(\frac{k}{k-1}\right).
\]
By taking the limit for the above inequalities, we obtain
\[
a_\infty = \liminf_{k\to\infty}a_k \le \liminf_{k\to\infty} k(x_\infty - x_k)\le \limsup_{k\to\infty} k(x_\infty - x_k) \le \limsup_{k\to\infty}a_k = a_\infty.
\]
Hence, we can conclude that $\displaystyle\lim_{k\to\infty}k(x_\infty -x_k) = a_\infty$. 

\section{Proof of Lemma~\ref{prop:updaterule}}\label{sec:proofOfprop}

Since $P_H(\bm{v}) = \bm{u}_*+\bm{B}\tilde{\bm{t}}$, the equality $(\bm{v}-\bm{u}_*-\bm{B}\tilde{\bm{t}})^T\bm{B} = \bm{0}$. It follows from the orthonormality of columns of $\bm{B}$ that $\tilde{\bm{t}} = \bm{B}^T(\bm{v}-\bm{u}_*)$. In addition, it follows from Lemma ~\ref{lemma:socp} that $\bm{v} = P_{\mathcal{K}}(\bm{u}) = \lambda_1\bm{e}_1$. 
Thus, we have $\tilde{\bm{t}} = \bm{B}^T(\lambda_1\bm{e}_1 - \bm{u}_*)$.

\section{Proof of Theorem~\ref{thm:example1}}\label{app:convrateEx1}

We analyze one iteration of APM. For this, we assume that APM updates $\tilde{\bm{t}} = (\tilde{x}, \tilde{y})$ from $\bm{t} = (x, y)$ with one iteration. If $y\neq 0$ and $x^2-(3y^2/2 + \sqrt{6}y) > 0$, then $\pm\bm{u}(\bm{t})\not\in \mathcal{K}$ and $\pm\bm{v}(\bm{t})\not\in\mathcal{K}$. The projection to $K$ is $(\lambda_1\bm{e}_1; \mu_1\bm{f}_1)$, where $\lambda_1$ and $\mu_1$ are the first spectral of $\bm{u}(\bm{t})$ and $\bm{v}(\bm{t})$, respectively. Also, $\bm{e}_1$ and $\bm{f}_1$ are the eigenvectors associated with $\lambda_1$ and $\mu_1$, respectively. 

We will prove that the convergence rate of the sequence $\{(X_k, Y_k)\}$ generated by the recurrence formula \eqref{recEx1} is $\Theta(k^{-1/6})$. For this, we define the two functions $S(X)$, $U(X)$ and the set $E(\delta)$ for a given $\delta >0$ as follows:
\begin{align*}
S(X) &:= \frac{X^2}{2} -\frac{3X^4}{8} + \frac{X^6}{2}, U(X) := X - \frac{X^7}{96}, \\
E(\delta) &:= \left\{
(X, Y)\in \mathbb{R}^2 : |Y- S(X)| < X^7, 0\le X\le \delta
\right\}.
\end{align*}

For simplicity, we define $X = x/\sqrt{3}$, $Y = y/\sqrt{6}$, $\tilde{X} =\tilde{x}/\sqrt{3}$ and $\tilde{Y} = \tilde{y}/\sqrt{6}$. If $Y\neq 0$ and $X^2 - (3Y^2+2Y) > 0$, then 
\begin{align}\label{recEx1}
\left\{
\begin{array}{lclcl}
\tilde{X} &=& f(X, Y) &:=& \frac{1}{6}\left(
3X-2 + \frac{2(X+1)^2+Y^2}{P(X, Y)} + \frac{(2Y+1)X}{Q(X, Y)}
\right), \\
\tilde{Y} &=& g(X, Y) &:=& \frac{1}{12}\left(
6Y-3 + \frac{(1+X)Y}{P(X, Y)} + \frac{2X^2+4Y^2+7Y+3}{Q(X, Y)}
\right), 
\end{array}
\right.
\end{align}
where $P(X, Y) = \sqrt{(X+1)^2 + Y^2}$ and $Q(X, Y) = \sqrt{X^2 + (Y+1)^2}$.

We choose $\delta >0$ such that $S(X)-X^7 >0$ for all $(X, Y)\in E(\delta)$. Then, we see that $Y > S(X)-X^7 > 0$. Thus, $\partial f/\partial Y >0$ and $\partial g/\partial Y >0$. Indeed, we have 
\begin{align*}
\frac{\partial f}{\partial Y} &= \frac{1}{6}\left(
\frac{Y^3}{P(X, Y)^3} + \frac{X(2X^2 + Y+1)}{Q(X, Y)^3}
\right), \\
\frac{\partial g}{\partial Y} &= \frac{1}{12}\left(6+
\frac{(X+1)^3}{P(X, Y)^3} + \frac{X^2(6Y+5) + 4(Y+1)^3}{Q(X, Y)^3}
\right).
\end{align*}
Since $X\ge 0$ and  $Y>0$ for all $(X, Y)\in E(\delta)$, $\partial f/\partial Y, \partial g/\partial Y > 0$. Thus, functions $f(X, Y)$ and $g(X, Y)$ are monotonically increasing with respect to $Y$ for a fixed $X$. Moreover, by taking a sufficiently small $\delta>0$, $S(X)$ and $U(X)$ are monotonically non-decreasing for all $0\le X\le \delta$. Indeed, we have $S(0)=U(0) = 0$ and $S^\prime (X) = X + O(X^3)$ and $U^\prime(X) = 1 + O(X^6)$ around $X=0$. 

As in the following lemma, by choosing a more appropriate $\delta>0$, the set $E(\delta)$ satisfies the properties necessary to guarantee the convergence rate of the APM. Indeed, this lemma implies that the sequence $\{(X_k, Y_k)\}$ is generated by the recurrence formula \eqref{recEx1} for all $k$ if $(X_0, Y_0)\in E(\delta)$. 

\begin{lemma}\label{lemma:convrateEx1}
There exists $\delta >0$ such that for all $(X, Y)\in E(\delta)$ and $(\tilde{X}, \tilde{Y})$ generated by \eqref{recEx1}, the following hold:
\begin{align}
\label{ineq1} X^2 - 3Y^2 -2Y &>0, Y>0,  \\
\label{ineq4} |\tilde{X}-U(X)| &< X^8, \\
\label{ineq5} |\tilde{Y} - S(U(X))| &< X^7, \\
\label{ineq3} (\tilde{X}, \tilde{Y}) &\in E(\delta).
\end{align}
\end{lemma}
\begin{proof}
Since we have $Y<S(X)+X^7$ for all $(X, Y)\in E(\delta)$,  
\[
X^2 - 3Y^2 -2Y > X^2 -3(S(X)+X^7)^2-2(S(X)+X^7) = \frac{X^6}{8} + O(X^7), 
\]
which implies $X^2-(3Y^2+2Y)>0$ for all $(X, Y)\in E(\delta)$ with a sufficiently small $\delta > 0$.

We define the functions $f_{\mathrm{L}}$, $f_{\mathrm{U}}$, $g_{\mathrm{L}}$ and $g_{\mathrm{U}}$ as follows:
\begin{align*}
f_{\mathrm{L}}(X) &:= f(X, S(X)-X^7), f_{\mathrm{U}}(X) :=f(X, S(X)+X^7), \\
g_{\mathrm{L}}(X) &:= g(X, S(X)-X^7), g_{\mathrm{U}}(X) :=g(X, S(X)+X^7).
\end{align*}
Since both function $f(X, Y)$ and $g(X, Y)$ are monotonically increasing with respect to $Y$ for a fixed $X$, we have 
\begin{align*}
f_{\mathrm{L}}(X)-U(X) &\le f(X, Y) - U(X)\le f_{\mathrm{U}}(X)-U(X), \\
g_{\mathrm{L}}(X)-S(U(X)) &\le g(X, Y) - S(U(X))\le g_{\mathrm{U}}(X)-S(U(X)). 
\end{align*}
Then these upper and lower bounds of $f(X, Y) - U(X)$ and $g(X, Y)-F(U(X))$ are 
\begin{align*}
f_{\mathrm{L}}(X)-U(X) &= -\frac{63X^8}{384} + O(X^9), f_{\mathrm{U}}(X)-U(X) = \frac{65X^8}{384} + O(X^9), \\
g_{\mathrm{L}}(X)-S(U(X)) &= -\frac{87X^7}{96} + O(X^8), g_{\mathrm{U}}(X)-S(U(X)) = \frac{89X^7}{96} + O(X^8). 
\end{align*}
Thus, there exists a sufficiently small $\delta > 0$ such that for all $(X, Y)\in E(\delta)$, 
\[
|f(X, Y) - U(X) | < X^8 \mbox{ and } |g(X, Y) - S(U(X))| < X^7.
\]

We prove that there exists a $\delta >0$ such that $(\tilde{X}, \tilde{Y})\in E(\delta)$ for all $(X, Y)\in E(\delta)$. First of all, since we can choose $\delta >0$ such that $U(X) - X^8 >0$ and $U(X) +X^8 < X$, it follows from the inequality $|f(X, Y)-U(X)|<X^8$ and the definition of $U$ that there exists a small $\delta >0$ such that $0\le \tilde{X}\le \delta$ for all $(X, Y)\in E(\delta)$. 

Since $d S/dX = X +O(X^3)$, $S(X)$ is monotonically increasing for $0\le X\le \delta$ with a sufficiently small $\delta > 0$. In addition, we have already seen that $g_{\mathrm{L}}(X)\le \tilde{Y}\le g_{\mathrm{U}}(X)$ and $U(X) -X^8\le \tilde{X}\le U(X)+X^8$. Thus we have
\begin{align*}
\tilde{Y}-S(\tilde{X}) &> g_{\mathrm{L}}(X) - S(U(X)+X^8) = -\frac{87X^7}{96} + O(X^8), \mbox{ and }\\
\tilde{Y}-S(\tilde{X}) &< g_{\mathrm{U}}(X) - S(U(X)-X^8) = \frac{89X^7}{96} + O(X^8). 
\end{align*}
Hence, there exists a sufficiently small $\delta>0$ such that $|\tilde{Y} - S(\tilde{X})| < X^7$ for all $(X, Y)\in E(\delta)$. 
\end{proof}

Theorem~\ref{thm:example1} follows from the next theorem.

\begin{theorem}\label{thm:ex41}
If $(X_0, Y_0)\in E(\delta)$, then the sequence $\{(X_k, Y_k)\}$ of the APM converges $(0, 0)$ with the rate $\Theta(k^{-1/6})$. 
\end{theorem}
\begin{proof}
It follows from the convergence property of APM that this sequence converges to $(0, 0)$. Moreover, it follows from \eqref{ineq4} that  
\[
\left|X_{k+1} - \left(X_k - \frac{X_k^7}{96}\right)\right| < X_k^8  \ (k=0, 1, \ldots).
\]
From Lemma~\ref{lemma1}, we obtain $X_k=\Theta(k^{-1/6})$. More precisely, we have $X_k \approx \sqrt[3]{4}k^{-1/6}$.  Also, it follows from the definition of $E(\delta)$ that we have 
\[
|Y_k - S(X_k)| < X_k^7 \ (k=0, 1, \ldots),
\]
which implies $Y_k = \Theta(k^{-1/3})$. Therefore the convergence rate of APM with $(X_0, Y_0)\in E(\delta)$ is $\Theta(k^{-1/6})$. 
\end{proof}

\section{Proof of Theorem~\ref{thm:example2}}\label{app:property}

For $(\bm{u}(\bm{t}); \bm{v}(\bm{t}))\in H$, we write the spectral decomposition of $\bm{u}(\bm{t})$ and $\bm{v}(\bm{t})$ as follows:
\[
\bm{u}(\bm{t}) = \lambda_1\bm{e}_1+\lambda_2\bm{e}_2, \bm{v}(\bm{t}) = \mu_1\bm{f}_1+\mu_2\bm{f}_2.
\]
For $\bm{t} = (x, y, z)$, we define $X=ax$, $Y=by$ and $Z=cz$. The projections of $(\bm{u}(\bm{t}); \bm{v}(\bm{t}))\in H$ to $K$ are as follows:
\begin{enumerate}[label=(Pr\arabic*)]
\item\label{Pr1} If $X\ge 0$ and $Z=0$, then $\bm{u}(\bm{t})\in\mathcal{K}$, and $\bm{v}(\bm{t}) \in \mathcal{K}$ or $\bm{v}(\bm{t}) \in -\mathcal{K}$. Thus, the projection is $(\bm{u}(\bm{t}), \bm{v}(\bm{t}))$ or $(\bm{u}(\bm{t}), \bm{0})$. Similarly, if $X<0$ and $Z=0$, then $\bm{u}(\bm{t}) \in -\mathcal{K}$ and, $\bm{v}(\bm{t}) \in \mathcal{K}$ or $\bm{v}(\bm{t}) \in -\mathcal{K}$. Thus, the projection is $(\bm{0}, \bm{v}(\bm{t}))$ or $(\bm{0}, \bm{0})$. Consequently, if $Z=0$, then we can conclude that the APM terminates in at most one iteration. 
\item\label{Pr3} If $Y\ge Z > 0$, then $\bm{v}(\bm{t})\in\mathcal{K}$. Thus the projection is $(\lambda_1\bm{e}_1; \bm{v}(\bm{t}))$.
\item\label{Pr4} If $Y\le Z < 0$, then $-\bm{v}(\bm{t})\in\mathcal{K}$. Thus the projection is $(\lambda_1\bm{e}_1; \bm{0})$.
\item\label{Pr5} Otherwise, the projection is $(\lambda_1\bm{e}_1; \mu_1\bm{f}_1)$.
\end{enumerate}

We analyze one iteration of APM. For this, we assume that APM updates $\tilde{\bm{t}} = (\tilde{x}, \tilde{y}, \tilde{z})$ from $\bm{t} = (x, y, z)$ with one iteration. For simplicity, we define $X = ax$, $Y = ay$, $Z=bz$, $\tilde{X} = a\tilde{x}$, $\tilde{Y} = a\tilde{y}$ and $\tilde{Z} = bz$. Also, we define $P(X, Z) = \sqrt{X^2 + 4Z^2}$ and $Q(Y, Z) = \sqrt{(Y-Z)^2 + 4Z^2}$. 

Then we obtain the following recurrence formulas for \ref{Pr3} to \ref{Pr5}: For all $(X, Y, Z)$, we have 
\begin{align}\label{X}
\tilde{X} &=\displaystyle\frac{1}{4}\left(X+P(X, Y)\right)\left(1+ \frac{X}{P(X, Z)}\right) = \frac{(X+P(X, Z))^2}{4P(X, Z)}.
\end{align}
\begin{description}
\item[\ref{Pr3}] If $Y\ge Z>0$, then 
\begin{align}
\label{Proj3}
    \left\{
\begin{array}{ccl}
\tilde{Y} &=& Y, \\ 
\tilde{Z}&=& \displaystyle \frac{3Z}{5} + \frac{1}{5}\left(X+P(X, Z)\right)\frac{Z}{P(X, Z)} = \frac{Z}{5}\left(4 + \frac{X}{P(X, Z)}\right).
\end{array}
    \right.
\end{align}
\item[\ref{Pr4}] If $Y\le Z < 0$, then 
\begin{align}
\label{Proj4}
    \left\{
\begin{array}{ccl}
\tilde{Y} &=& 0, \\
\tilde{Z}&=&\displaystyle\frac{1}{5}\left(X + P(X, Z)\right)\frac{Z}{P(X, Z)} = \frac{Z}{5}\left(1 + \frac{X}{P(X, Z)}\right). 
\end{array}
    \right.
\end{align}
\item[\ref{Pr5}] Otherwise, 
\begin{align}
\label{Proj5}
    \left\{
\begin{array}{ccl}
\tilde{Y} &=&\displaystyle\frac{1}{4}\left(
Y+Z+Q(Y, Z)\right)\left(1+\frac{Y-Z}{Q(Y, Z)}\right) = \frac{Y}{2}\left(1+\frac{Y-Z}{Q(Y, Z)}\right)+\frac{Z^2}{Q(Y, Z)}, \\
\tilde{Z} &=&\displaystyle\frac{1}{5}\left(X+P(X, Z)\right)\frac{Z}{P(X, Z)}+\displaystyle\frac{1}{20}\left(Y+Z +Q(Y, Z)\right)\left(1+\frac{5Z-Y}{Q(Y, Z)}\right)\\
&=&Z \displaystyle\left(\frac{1}{2} + \frac{X}{5P(X, Z)} + \frac{Y+5Z}{10Q(Y, Z)}\right).
\end{array}
    \right.
\end{align}
\end{description}

The transition of the three recurrence formulas can be clarified from the following lemma. 
\begin{lemma}\label{lemma:property}
The following hold:
\begin{enumerate}[label=(L\arabic*)]
\item \label{L0X} $\tilde{X}\ge X$ and $\tilde{X}\ge 0$. 
\item\label{L2Pr3} If $Y\ge Z > 0$, then $Z > \tilde{Z} > 0$ and $\tilde{Y}=Y>0$. 
\item\label{L3Pr4} If $Y\le Z < 0$, then $\tilde{Z} < 0$ and $\tilde{Y} = 0$. 
\item\label{L5Pr5} If $Y < Z$ and $Z> 0$, then $\tilde{Z} > 0$, $\tilde{Y} > 0$ and $\tilde{Y}> Y$. 
\item\label{L4Pr5} If $Y>Z$ and $Z< 0$,  then $\tilde{Y}>0>\tilde{Z}$ and $\tilde{Y} >Y$. \end{enumerate}
\end{lemma}
\begin{proof}
We use in this proof the following inequalities: for all $A, B\in\mathbb{R}$, 
\begin{align*}
|A/\sqrt{A^2+B^2}|&\le 1, \\
A+B + \sqrt{A^2+B^2} &\left\{
\begin{array}{ll}
> 0 & (A, B\ge 0 \mbox{ except for } A=B=0, \mbox{ or } AB < 0), \\
<0 & (A, B < 0). 
\end{array}
\right.
\end{align*}

\ref{L0X} : We consider the recurrence formula on $X$ in \eqref{X}. It is clear that $\tilde{X}\ge 0$. The inequality $\tilde{X}\ge X$ follows from 
\begin{align*}
\tilde{X}-X &= \frac{1}{4}\left(-2X + P(X, Z) + \frac{X^2}{P(X, Z)}\right)= \frac{(X - P(X, Z))^2}{4P(X, Z)}\ge 0. 
\end{align*}

\ref{L2Pr3} : Since $Z>0$, it is clear that $\tilde{Z} > 0$. In addition, the inequality $Z> \tilde{Z}$ follows from the inequality 
$Z - \tilde{Z}=\frac{Z}{5}\left(
1 - \frac{X}{P(X, Z)}
\right)\ge 0$. 

\ref{L3Pr4} : Since $Z<0$, we  can conclude $\tilde{Z} < 0$.

\ref{L5Pr5} : We have $Y+Z+Q(Y, Z) = Y -Z + 2Z + Q(Y, Z)$. Since $(Y-Z) Z < 0$, we have $Y+Z+Q(Y, Z) > 0$, which concludes $\tilde{Y}> 0$. In addition, since $-Y-Z + Q(Y, Z) =-(Y-Z)-2Z+Q(Y, Z)>0$ and $1-(Y-Z)/Q(Y, Z) > 0$, we have 
\[
\tilde{Y} - Y = \frac{1}{2}\left(
-Y-Z + Q(Y, Z)\right)\left(1 - \frac{Y-Z}{Q(Y, Z)}\right)> 0, 
\]
which implies $\tilde{Y}>Y$. In addition, using $Y-Z <0$ and $Z>0$, we obtain $|X/P(X, Z)| \le 1$ and $-1< (Y+5Z)/Q(Y, Z) < 3$. Consequently, we obtain
\[
\frac{Z}{5} = Z\left(\frac{1}{2}-\frac{1}{5}-\frac{1}{10}\right)< \tilde{Z} < Z\left(\frac{1}{2} + \frac{1}{5} + \frac{3}{10}\right) = Z, 
\]
which imply $\tilde{Z} > 0$. 

\ref{L4Pr5} :  Since $Z<0$ and $(Y-Z)Z< 0$, the inequalities $\tilde{Y} > Y$, $\tilde{Y}>0$ follows from the same proof in \ref{L5Pr5}. In addition, using $Y-Z >0$ and $Z<0$, we obtain $|X/P(X, Z)| \le 1$ and $-3< (Y+5Z)/Q(Y, Z) < 1$. Consequently, we obtain
\begin{align}\label{eqZ}
\frac{4Z}{5}=Z\left(\frac{1}{2} + \frac{1}{5} + \frac{1}{10}\right)< \tilde{Z} <  Z\left(\frac{1}{2}-\frac{1}{5}-\frac{3}{10}\right) = 0, 
\end{align}
which imply $\tilde{Z} < 0$. 
\end{proof}

We consider the sequence $\{(X_k, Y_k, Z_k)\}$ generated by these recurrence formulas \eqref{Proj3} to \eqref{Proj5}. It follows from the convergence property of APM that this sequence convergences $(X_\infty, Y_\infty, Z_\infty)$, where $X_\infty, Y_\infty\ge 0$, $Z_\infty= 0$. Lemma~\ref{lemma:property} implies that  it is sufficient to consider the sequence $\{(X_k, Y_k, Z_k)\}$ generated by \eqref{Proj3} or \eqref{Proj5}. Indeed, we observe following from Lemma \ref{lemma:property}: 
\begin{itemize}
\item \ref{L0X} implies that $\{X_k\}$ is the monotonically non-decreasing sequence and $X_k\ge 0$ for $k=1, 2, 3, \ldots$. In particular, if $Z_0\neq 0$, then $X_1 >0$. Since $\{X_k\}$ is the monotonically non-decreasing sequence,  we have $X_k>0$ for $k=1, 2, \ldots$. If $Z_0=0$, then APM finds a point in the intersection of $H$ and $K$ with at most one iteration. Furthermore, we see from \ref{L2Pr3} to \ref{L4Pr5} that $\tilde{Y}>0$ and $\tilde{Y} > Y$ hold for all cases. In other words, the sequence $\{Y_k\}$ generated by APM is monotonically increasing and $Y_k\ge 0$ for all $k=1, 2, \ldots$. 

\item \ref{L2Pr3} implies that if $Y_0\ge Z_0 > 0$, then the recurrence formula \eqref{Proj3} is applied for $(X_k, Y_k, Z_k)$ with $k=1, 2, \ldots$. \ref{L5Pr5} implies that if $Y_0 < Z_0$ and $Z_0 > 0$, then $Y_1 > 0$. Thus,  $Y_k > 0$ for all $k$ follows from \ref{L2Pr3} and \ref{L5Pr5}. In addition, $\{(X_k, Y_k, Z_k)\}$ converges $(X_\infty, Y_\infty, Z_\infty)$ with $X_\infty, Y_\infty\ge 0$, $Z_\infty=0$. In particular, $Y_\infty > 0$, and thus there exists $\tilde{k}\in\mathbb{N}$ such that $Y_k \ge  Z_k$ for all $k\ge \tilde{k}$. This implies that the sequence $\{(X_k, Y_k, Z_k)\}$ is generated by \eqref{Proj3} for all $k\ge \tilde{k}$. 
 
To prove the inequality $Y_\infty > 0$, we suppose that $Y_k < Z_k$ for all $k$. Then the sequence $\{(X_k, Y_k, Z_k)\}$ is generated by \eqref{Proj5} when $Y_0 < Z_0$ and $Z_0 > 0$. It follows from \ref{L5Pr5} of Lemma ~\ref{lemma:property} that the sequence $\{Y_k\}$ satisfies $Y_{k+1}> Y_k\ge  Y_1>0$. This implies that $Y_\infty >0$. In contrast, $Y_\infty \le 0$ follows from the assumption, which is a contradiction. Hence $Y_{\tilde{k}} \ge Z_{\tilde{k}}$ holds at some $\tilde{k}$. Consequently, it follows from \ref{L2Pr3} of Lemma~\ref{lemma:property} that  the sequence $\{(X_k, Y_k, Z_k)\}$ is generated by \eqref{Proj3} for all $k\ge \tilde{k}$.

\item \ref{L3Pr4} implies if $Y_0 \le Z_0 < 0$, then $Z_1<0$ and $Y_1=0$, and thus $(X_2, Y_2, Z_2)$ is generated by the recurrence formula \eqref{Proj5}. \ref{L4Pr5} implies that if $Y_0 > Z_0$ and $Z_0 < 0$, then the recurrence formula \eqref{Proj5} is applied for $(X_k, Y_k, Z_k)$ with $k=1, 2, \ldots$. 
It follows from \ref{L3Pr4} and \ref{L4Pr5} that if $Z_0 <0$, then $Z_k < 0 \le Y_k$ holds for all $k$ and that the recurrence formula \eqref{Proj5} is applied for $(X_k, Y_k, Z_k)$ with $k=2, 3,\ldots$. 

\end{itemize}

Figures \ref{fig:ex42-1} and \ref{fig:ex42-2} summarize the observations of Lemma~\ref{lemma:property}. Figure ~\ref{fig:ex42-1} displays that if the recurrence formula \eqref{Proj4} is applied for $(X_0, Y_0, Z_0)$, then the recurrence formula \eqref{Proj5} is applied for $(X_1, Y_1, Z_1)$. In addition, if $Y_0\ge Z_0 > 0$ and $Z_0 \le 0$, then the recurrence formula \eqref{Proj3} is applied after some iterations. 

Figure~\ref{fig:ex42-2} displays the recurrence formulas applied in the APM by using $(Y_0, Z_0)$. If $Y_0\ge  Z_0 > 0$, then the sequence of APM is generated by \eqref{Proj3} and converges to a point on the half-line $Y\ge 0$ and $Z = 0$. If $Y_0\le Z_0 < 0$, then the sequence is generated by \eqref{Proj5} after one iteration and converges to a point on the half-line. If $Y_0 < Z_0$ and $Z_0 <0$, then \eqref{Proj5} is applied for the first few iterations, and \eqref{Proj3} is applied after the iterations. If $Y_0\le Z_0 < 0$, then the sequence is generated by \eqref{Proj5} after one iteration and converges to a point on the half-line. 

\begin{figure}[ht]
\centering
    \centering
    \begin{tikzpicture}[->,auto,node distance=2.5cm,thick, scale=0.50]
  \node[circle,draw] (1) {(15A)};
  \node[circle,draw, right of=1] (2) {\eqref{Proj3}};
  \node[circle,draw, right of=2] (3) {\eqref{Proj4}};
  \node[circle,draw, right of=3] (4) {(15B)};

  \path[every node/.style={sloped,anchor=south,auto=false}]
    (2) edge [loop above] node {} (2)
    (4) edge [loop above] node {} (4);

  \path
    (3) edge [bend left] node {} (4)
    (1) edge [bend left, dotted] node {} (2);

\end{tikzpicture}
\caption{Transition diagram of the recurrence formulas \eqref{Proj3}, \eqref{Proj4} and \eqref{Proj5}: (15A) means the recurrence formula \eqref{Proj5} with $Y_0 < Z_0$ and $Z_0>0$, and (15B) means \eqref{Proj5} with $Y_0 > Z_0$ and $Z_0<0$. The solid line from \eqref{Proj4} to (15B) means that the recurrence formula changes in one iteration. A dotted line  means that the recurrence formula changes in a few iterations. }
\label{fig:ex42-1}
\centering
\begin{tikzpicture}[scale=2.0]
  \filldraw[red!30] (1.2,1.2) -- (0,0) -- (1.2,0) -- (1.2, 1.2) --cycle;
  \filldraw[blue!30] (-1.2,-1.2) -- (0, 0) -- (-1.2, 0) -- (-1.2, -1.2) -- cycle;
  \filldraw[green!30] (1.2,1.2) -- (-1.2,1.2) -- (-1.2,0) -- (0, 0) -- cycle;
  \filldraw[green!30] (-1.2,-1.2) -- (0,0) -- (1.2,0) -- (1.2,-1.2) -- cycle;
  \node at (0.9, 0.3){\eqref{Proj3}};
  \node at (-0.9, -0.3){\eqref{Proj4}};
  \node at (-0.6, 0.6){(15A)};
  \node at (0.6, -0.6){(15B)};
  \node at (0.1, -0.1){$O$};
  \draw[->] (0, -1.2,0) -- (0, 1.2) node[above] {$Z$};
  \draw[ultra thick, ->] (0,0) -- (1.2, 0) node[right] {$Y$};
  \draw[very thick,dotted, -] (-1.2, 0) -- (0,0) node[above] {};
\end{tikzpicture}
\caption{If the initial point $(Y_0, Z_0)$ is at the red area, then the recurrence formula \eqref{Proj3} is applied. If $(Y_0, Z_0)$ is at the blue area, then the recurrence formula \eqref{Proj4} is applied. If $(Y_0, Z_0)$ is on the dotted line, the next point is $(0, 0)$. Otherwise, the recurrence formula \eqref{Proj5} is applied. The definitions of (15A) and (15B) are provided in the caption of Figure~\ref{fig:ex42-1}.}
\label{fig:ex42-2}
\end{figure}

The following theorem shows the convergence property of APM for this example. Theorem~\ref{thm:example2} follows from this theorem. 
\begin{theorem}\label{thm:example2F}
Consider the sequence $\{(X_k, Y_k, Z_k)\}$ generated by  applying APM. Let $(X_\infty, Y_\infty, Z_\infty)$ be the convergence point of the sequence. Then, the following hold. 
\begin{enumerate}[label=(C\arabic*)]
\item\label{CC0} If $Z_0=0$, then $(X_0, Y_0, Z_0)\in H\cap K$ or $(X_1, Y_1, Z_1)\in H\cap K$.  
\item\label{CC1} If $Z_0 >0$, then there exists $\tilde{k}\in\mathbb{N}$ such that the sequence $\{(X_k, Y_k, Z_k)\}$ is generated by \eqref{Proj3} for all $k\ge \tilde{k}$, and thus we have $X_k-X_\infty = \Theta(k^{-1})$, $|Z_k| = \Theta(k^{-1/2})$ and  $Y_k$ is constant. 
\item\label{CC2} If $Z_0 <0$, then there exists $\tilde{k}\in\mathbb{N}$ such that  the sequence $\{(X_k, Y_k, Z_k)\}$ is generated by \eqref{Proj5} for all $k\ge \tilde{k}$, and thus we have $X_k-X_\infty = \Theta((4/5)^{4k})$, $|Z_k| = \Theta((4/5)^{k})$ and $Y_\infty -Y_k=\Theta((4/5)^{3k})$. 
\end{enumerate}
\end{theorem}
\begin{proof}
\ref{CC0}: We have already mentioned in \ref{Pr1} that this statement is true. 

\ref{CC1}: We have already mentioned that if $Z_0 >0$, then  there exists $\tilde{k}\in\mathbb{N}$ such that  $Y_k \ge Z_k > 0$ holds for all $k \ge \tilde{k}$, and thus $(X_k, Y_k, Z_k)$ is generated by the recurrence formula \eqref{Proj3} for all $k\ge \tilde{k}$. Consequently, we have $Y_{k+1} = Y_k$ for all $k\ge \tilde{k}$, and thus $Y_k = Y_{\tilde{k}}$.

For simplicity, we assume $\tilde{k}=0$. 

Define $f(x, z) := \displaystyle\frac{4}{5P(x, z)\left(x+P(x, z)\right)}$ and $P_k = P(X_k, Z_k)$. Then the recurrence formula on $Z_k$ in \eqref{Proj3} can be rewritten as 
\begin{align*}
Z_{k+1} &= Z_k\left(1-\frac{1}{5} + \frac{X_k}{5P_k}\right) = Z_k\left(1-\frac{P_k-X_k}{5P_k}\right) =Z_k\left(1-Z_k^2 f(X_k, Z_k)\right).
\end{align*}
It follows from \ref{L4} of Lemma~\ref{lemma1} that $\lim_{k\to\infty}1/(kZ_k^2) = 2f(X_\infty, 0) = 4/(5X_\infty^2)$, and thus $Z_k = \Theta(k^{-1/2})$. 

To see the convergence rate of $\{X_k\}$, we define the sequence $\{a_k\}$ with 
\[
a_k = \frac{4k^2 Z_k^4}{P_k (P_k +X_k)^2} \ (k=0, 1, \ldots). 
\]
Then we have 
\begin{align*}
X_{k+1} - X_k &= \frac{(P_k -X_k)^2}{4P_k}= \frac{4Z_k^4}{P_k (P_k +X_k)^2} = \frac{a_k}{k^2}. 
\end{align*}
Moreover, we have $\displaystyle\lim_{k\to\infty}a_k = 25X_\infty/16$. Thus, it follows from \ref{L5} of Lemma~\ref{lemma1} that $\displaystyle\lim_{k\to\infty}k(X_\infty-X_k) = 25X_\infty/16$ which implies $X_\infty - X_k = \Theta(k^{-1})$. 

\ref{CC2}: As we have already mentioned, if $Y_0\le Z_0 < 0$, then the sequence $\{(X_k, Y_k, Z_k)\}$ with $k=2, 3, \ldots$ is generated by the recurrence formula \eqref{Proj5}. Thus, we assume $Y_0 > Z_0$ and $Y_0< 0$ for simplicity.  

We define $Q_k = Q(Y_k, Z_k)$ and $a_k = \frac{5}{4}\left(\frac{1}{2} + \frac{X_k}{5P_k} + \frac{Y_k+5Z_k}{10Q_k}\right)$. Then, we have 
\begin{align}
\label{Wk0}Z_{k+1} &=
\frac{4}{5}a_kZ_k. 
\end{align}
In addition, the inequality $0 < a_k <  1$ follows from  \eqref{eqZ}, $Y_k > Z_k$ and $Z_k < 0$. Hence, we obtain  $|Z_{k+1}|\le (4/5)|Z_k|$ for all $k=0, 1, \ldots$, and thus 
$|Z_k| \le (4/5)^k|Z_0|$ for all $k=0, 1, \ldots$. 

Since we obtain $\displaystyle Z_k \left(\frac{5}{4}\right)^k = Z_0\prod_{\ell=0}^{k-1}a_{\ell}$ from \eqref{Wk0},  
we will prove the infinite product $\displaystyle\prod_{\ell=0}^{\infty}a_{\ell}$ converges. Then, we can conclude $|Z_k| = \Theta((4/5)^k)$. To this end, we consider the sequence $\{a_k\}$. It should be noted that we have $P_k-X_k = \frac{4Z_k^2}{P_k+X_k}$ and $Q_k-(Y_k-Z_k) = \frac{4Z_k^2}{Q_k+Y_k-Z_k}$. We define 
\[
b_k = \frac{|Z_k|}{P_k(P_k+X_k)} + \frac{|Z_k|}{2Q_k(Q_k+Y_k-Z_k)} + \frac{3}{4Q_k}.
\]
It should be noted that $b_k > 0$ for all $k$, and $\{b_k\}$ is bounded because the limit of $\{b_k\}$ is $3/(4Y_\infty)$. We can rewrite $a_k$ by $b_k$ and $Z_k$ as follows:
\begin{align*}
a_k &= \frac{5}{4}\left(
\frac{1}{2} + \frac{1}{5} -\frac{P_k-X_k}{5P_k} + \frac{1}{10} -\frac{Q_k-Y_k+Z_k}{10Q_k} + \frac{6Z_k}{10Q_k}
\right)\\
&=1 - \frac{P_k-X_k}{4P_k} - \frac{Q_k-Y_k+Z_k}{8Q_k} +\frac{3Z_k}{4Q_k}=1-|Z_k|b_k. 
\end{align*}
Since the infinite sum $\sum_{k=0}^\infty |Z_k|$ converges and $\{b_k\}$ is bounded, $\sum_{k=0}^\infty|Z_k|b_k$ converges, and thus $\prod_{k=0}^\infty a_k$ also converges. This implies $|Z_k| = \Theta((4/5)^{k})$.

We focus on $Y_k$. Using equations $Q_k-Y_k+Z_k = \frac{4Z_k^2}{Q_k+Y_k-Z_k}$,  we obtain
\begin{align*}
\nonumber Y_{k+1} 
&= \frac{Y_k}{2}\left(2 - \frac{Q_k-Y_k+Z_k}{Q_k}\right)+ \frac{Z_k^2}{Q_k} = Y_k + \frac{Z_k^2}{Q_k(Q_k+Y_k-Z_k)}\left(Q_k-Y_k-Z_k\right)\\
\nonumber&= Y_k - \frac{2Z_k^3}{Q_k(Q_k+Y_k-Z_k)} + \frac{4Z_k^4}{Q_k(Q_k+Y_k-Z_k)^2}. 
\end{align*}
From the last equality, we obtain
\[
Y_\infty - Y_k = \sum_{\ell = k}^\infty\frac{-2Z_\ell^3}{Q_\ell(Q_\ell+Y_\ell-Z_\ell)} + \sum_{\ell = k}^\infty\frac{4Z_\ell^4}{Q_\ell(Q_\ell+Y_\ell-Z_\ell)^2}. 
\]
We prove that the sequence $\{\displaystyle (5/4)^{3k}(Y_\infty-Y_k)\}$ converges to a nonzero constant. This implies $Y_\infty-Y_k = \Theta((4/5)^{3k})$. For this, we define $\displaystyle u_{\ell} = \frac{(5/4)^{3\ell}(-2Z_{\ell}^3)}{Q_\ell(Q_\ell + Y_\ell-Z_\ell)}$ and $\displaystyle v_{\ell} = \frac{4(5/4)^{4\ell}Z_\ell^{4}}{Q_\ell(Q_\ell+Y_\ell-Z_\ell)^2}$ for all $\ell$. We remark that the sequences $\{u_{\ell}\}$ and $\{v_{\ell}\}$ are convergent because we have already proved $|Z_k| = \Theta((4/5)^k)$. Let $u_\infty$ and $v_\infty$ be the limits of $\{u_\ell\}$ and $\{v_\ell\}$, respectively. In particular, $u_{\infty}\neq 0$. Then we have 
\begin{align*}
\left(\frac{5}{4}\right)^{3k}\left(Y_\infty-Y_k\right) &= \sum_{\ell = k}^\infty \left(\frac{4}{5}\right)^{3\ell-3k}u_{\ell} + \sum_{\ell = k}^\infty \left(\frac{4}{5}\right)^{4\ell-3k}v_{\ell}\\
&= \sum_{\ell=k}^\infty \left(\frac{4}{5}\right)^{3\ell-3k}u_\infty + \sum_{\ell=k}^\infty \left(\frac{4}{5}\right)^{3\ell-3k}(u_\ell - u_\infty) + \left(\frac{4}{5}\right)^{k}\sum_{\ell=k}^\infty\left(\frac{4}{5}\right)^{4\ell-4k}v_{\ell}\\
&=\frac{125}{61}u_\infty + \sum_{\ell=k}^\infty \left(\frac{4}{5}\right)^{3\ell-3k}(u_\ell - u_\infty) + \left(\frac{4}{5}\right)^{k}\sum_{\ell=k}^\infty\left(\frac{4}{5}\right)^{4\ell-4k}v_{\ell}.
\end{align*}
Since we have $u_\ell \to u_\infty$ and $v_\ell \to v_\infty$ $(\ell\to\infty)$, then the second and third terms in the above equation converge to $0$ as $k\to\infty$. Thus, we obtain the desired result on $Y_\infty-Y_k$. 

Finally, we can prove $X_\infty-X_k = \Theta((4/5)^{4k})$. Indeed, we have 
\[
X_{k+1}-X_k = \frac{Z_k^4}{4 P_k(P_k+X_K)^2}. 
\]
Hence we obtain
\begin{align*}
\left(\frac{5}{4}\right)^{4k}(X_\infty - X_k) &= \sum_{\ell=k}^\infty\frac{(5/4)^{4k}Z_\ell^4}{4 P_\ell(P_\ell+X_\ell)^2}= \sum_{\ell=k}^\infty\left(\frac{4}{5}\right)^{4\ell-4k}\frac{(5/4)^{4\ell}Z_\ell^{4}}{4P_\ell(P_\ell+X_\ell)^2}. 
\end{align*}
By applying a similar manner to the proof of the convergence of the sequence $\{(5/4)^{3k}(Y_\infty-Y_k)\}$, we can prove that the sequence $\{(5/4)^{4k}(X_\infty-X_k)\}$ converges to a nonzero constant. Therefore, we obtain $X_\infty-X_k = \Theta((4/5)^{4k})$
\end{proof}
\begin{remark}
In the proof of \ref{CC1} in Theorem ~\ref{thm:example2F}, we obtained
\begin{align*}
\lim_{k\to\infty} k\left(X_\infty -X_k\right) &= \frac{25}{16}X_\infty \mbox{ and }
\lim_{k\to\infty} \sqrt{k}Z_k =\sqrt{\frac{5}{4}}X_\infty,  
\end{align*}
which provide more precise convergence rates of $\{X_k\}$ and $\{Z_k\}$. Indeed, we obtain $X_\infty-X_k\approx (25X_\infty/16)k^{-1}$ and $Z_k\approx\sqrt{5/4}X_\infty k^{-1/2}$. 

Similarly, we obtain more precise convergence rates of $\{X_k\}$, $\{Y_k\}$ and $\{Z_k\}$ in \ref{CC2} of  Theorem~\ref{thm:example2F}. Indeed, we define $C := -\lim_{k\to\infty}Z_k(5/4)^{k}$. Then, we have $u_\infty = C^3/Y_\infty^2$, and thus
\begin{align*}
\lim_{k\to\infty}\left(Y_\infty-Y_k\right)\left(\frac{5}{4}\right)^{3k}&= \frac{125C^3}{61Y_\infty^2} \mbox{ and }\lim_{k\to\infty}\left(X_\infty-X_k\right)\left(\frac{5}{4}\right)^{4k} = \frac{625C^4}{5904 X_\infty^3}. 
\end{align*}
Consequently, we obtain 
\[
X_\infty -X_k\approx\left(\frac{4}{5}\right)^{4k}\frac{625C^4}{5904 X_\infty^3}, Y_\infty -Y_k\approx \left(\frac{4}{5}\right)^{3k}\frac{125C^3}{61Y_\infty^2} \mbox{ and } Z_k \approx-C\left(\frac{4}{5}\right)^k. 
\]
\end{remark}

\end{document}